\documentclass[12pt,reqno,UTF8]{amsart}
\usepackage{amsfonts,verbatim, amscd,mathrsfs, color, array}
\usepackage{amsmath,latexsym,amssymb,amsbsy, amsthm}
\usepackage{graphicx}

\newcommand{\beq}{\begin{equation}}
	\newcommand{\eeq}{\end{equation}}
\newcommand{\ben}{\begin{eqnarray}}
	\newcommand{\een}{\end{eqnarray}}
\newcommand{\beno}{\begin{eqnarray*}}
	\newcommand{\eeno}{\end{eqnarray*}}

\usepackage[hmargin=3.0cm, vmargin=4.0cm]{geometry}
\usepackage{bbm}
\usepackage{pdfsync}
\usepackage{epstopdf}
\usepackage{cite}
\usepackage[english]{babel}
\usepackage{multirow}
\usepackage[font=bf,aboveskip=15pt]{caption}
\usepackage[toc,page]{appendix}
\usepackage[colorlinks,
linkcolor=red,
anchorcolor=blue,
citecolor=blue]{hyperref}
\usepackage{enumerate}
\usepackage{bookmark}
\DeclareFontShape{OT1}{cmr}{bx}{sc}{<-> cmbcsc10}{}
\usepackage{float}
\usepackage{ulem}
\usepackage{syntonly}
\usepackage{mathtools}
\usepackage{bm}
\usepackage{bbding}
\usepackage{wasysym}
\usepackage{diagbox}

\usepackage{upgreek}

\newtheorem{lem}{Lemma}[section]
\newtheorem{prop}{Proposition}[section]
\newtheorem{thm}{Theorem}[section]

\newtheorem{rmk}{Remark}[section]
\newcommand{\bremark}{\begin{remark} \em}
	\newcommand{\eremark}{\end{remark} }

\allowdisplaybreaks[4]

\numberwithin{equation}{section}
\DeclareMathOperator{\supp}{supp}
\hypersetup{
	colorlinks=true,
	linkcolor=black, 
	citecolor=black, 
	urlcolor=black 
}
\usepackage{enumitem}
\usepackage{amsthm}

\begin{document}
	
	\title[]{Normalized solutions to critical Choquard systems with linear and nonlinear couplings}

\author[]{Wenliang Pei and Chonghao Deng}
\address{\noindent
	State Key Laboratory of Mathematical Sciences, Academy of Mathematics and Systems Science, Chinese Academy of Sciences, Beijing 100190, China;
	School of Mathematical Sciences, University of Chinese Academy of Sciences, Beijing 100049, China
}
\email{peiwenliang@amss.ac.cn}

\address{\noindent
	School of Statistics, University of International Business and Economics, Beijing 100029, China
}
\email{dengchonghao@amss.ac.cn}


	\begin{abstract}
		We consider the critical Choquard system with both linear and nonlinear couplings
		\begin{equation*}
			\left\{
			\begin{aligned}
				-\Delta v_1 + \mu_1 v_1 &= ( I_\omega * |v_1|^{2_\omega^*} ) |v_1|^{2_\omega^* -2} v_1 + \theta p( I_\omega * |v_2|^q)|v_1|^{p-2} v_1 + \varepsilon v_2, \quad in \,\, \mathbb{R}^N, \\
				-\Delta v_2 + \mu_2 v_2 &= ( I_\omega * |v_2|^{2_\omega^*} ) |v_2|^{2_\omega^* -2} v_2 + \theta q( I_\omega * |v_1|^p)|v_2|^{q-2} v_2 + \varepsilon v_1 , \quad in \,\, \mathbb{R}^N ,\\
				\int_{\mathbb{R}^N} v_1^2 = \alpha_1^2\, &,  \int_{\mathbb{R}^N} v_2^2 = \alpha_2^2,
			\end{aligned}
			\right.
		\end{equation*}
		where $N=3\,\, \text{or} \,\, 4$, $\alpha_1,\alpha_2 > 0 $, $\theta > 0 $, $2_{\omega,*} :=\frac{N+\omega}{N} <p,q<2_\omega^*:=\frac{N+\omega}{N-2}$, $\varepsilon>0$, $0<\omega<N$, $I_\omega: \mathbb{R}^N \to  \mathbb{R}$ represents the Riesz potential. For the $L^2$-subcritical case $p+q<\frac{2N+2\omega+4}{N}$, we utilize the Ekeland's variational principle to obtain the existence of a positive normalized ground state for the system as $0<\theta<\theta_0,\;0<\varepsilon<\varepsilon_*$. For the $L^2$-supercritical case $p+q>\frac{2N+2\omega+4}{N}$, we apply variational methods to establish the existence of a positive normalized ground state for the system as $\theta>\theta_*,\;0<\varepsilon<\overline{\varepsilon}$.
	\end{abstract}
	\maketitle

	\section{Introduction}
	
 \par In this paper, we consider the critical Choquard system with both linear and nonlinear couplings
\begin{equation}
	\left\{
	\begin{aligned}
		-\Delta v_1 + \mu_1 v_1 &= ( I_\omega * |v_1|^{2_\omega^*} ) |v_1|^{2_\omega^* -2} v_1 + \theta p( I_\omega * |v_2|^q)|v_1|^{p-2} v_1 + \varepsilon v_2, \quad \text{in} \,\, \mathbb{R}^N, \\
		-\Delta v_2 + \mu_2 v_2 &= ( I_\omega * |v_2|^{2_\omega^*} ) |v_2|^{2_\omega^* -2} v_2 + \theta q( I_\omega * |v_1|^p)|v_2|^{q-2} v_2 + \varepsilon v_1 , \quad \text{in} \,\, \mathbb{R}^N ,\\
		\int_{\mathbb{R}^N} v_1^2 = \alpha_1^2\, &,  \int_{\mathbb{R}^N} v_2^2 = \alpha_2^2,
	\end{aligned}
	\right.
	\label{system1.1}
\end{equation}
where $N=3\,\, \text{or} \,\, 4$, $\alpha_1,\alpha_2 > 0 $, $\theta > 0 $, $2_{\omega,*} :=\frac{N+\omega}{N} <p,q<2_\omega^*:=\frac{N+\omega}{N-2}$, $\varepsilon>0$, $0<\omega<N$, the Riesz potential $I_\omega: \mathbb{R}^N \to \mathbb{R}$ is given by
\begin{equation*}
	I_\omega(x)=\frac{\Gamma(\frac{N-\omega}{2})}{2^\omega \pi^{\frac{N}{2}}\Gamma(\frac{\omega}{2})}\cdot \frac{1}{|x|^{N-\omega}}, \qquad \forall x \in \mathbb{R}^N\setminus \{0\}.
\end{equation*} 

	\par The system \eqref{system1.1} arises from the Schr\"{o}dinger system with both linear and nonlinear couplings as follows
\begin{equation}
	\left\{
	\begin{aligned}
		-i\partial_t \Phi_1-\varepsilon \Phi_2&=\Delta \Phi_1+(H(x)*|\Phi_1|^a)|\Phi_1|^{a-2}\Phi_1+\theta p(H(x)*|\Phi_2|^q)|\Phi_1|^{p-2}\Phi_1,\\
		-i\partial_t \Phi_2-\varepsilon \Phi_1&=\Delta \Phi_2+(H(x)*|\Phi_2|^b)|\Phi_2|^{b-2}\Phi_2+\theta q(H(x)*|\Phi_1|^p)|\Phi_2|^{q-2}\Phi_2,
	\end{aligned}
	\right.
	\label{system1.2}
\end{equation}
	which originates from Bose-Einstein condensations\cite{Anderson1995,Deconinck2004,Timmermans1998,zhang2023normalizedsolutionscriticalchoquard}. Physically, $\varepsilon$ signifies the radio-frequency coupling strength. The solution $\Phi_1,\Phi_2:\mathbb{R}^+ \times \mathbb{R}^N \to \mathbb{C}$  denote the condensate wave functions. The non-negative response function $H$ carries the information regarding the self-interaction of particles. For the coupling constant $\theta$, the condition $\theta >0$ implies the attractive interaction, while  $\theta<0$ indicates the repulsive interaction. The system is characterized by mass conservation. Namely, 
	\[\int_{\mathbb{R}^N}|\Phi_1(t,x)|^2dx,~~\int_{\mathbb{R}^N}|\Phi_2(t,x)|^2dx\]
	are not related to $t\in \mathbb{R}^+$. Indeed, $|\Phi_i(t,\cdot)|_2$ denotes the quantity of particles, $i=1,2$. The system is also applied in nonlinear optics\cite{Menyuk1987}.
	
	The standing wave solutions of this system are the focus of related researches, that is $\Phi_j(t,x)=e^{i\mu_j t}v_j(x)$,  $j=1,2$, as $\mu_1=\mu_2$. Meanwhile, let $H(x)=\delta(x)$, with $\delta$ being the Dirac delta function defined as a distribution satisfying
	\[
	\int_{\mathbb{R}^N} \delta(x)v(x)\,dx = v(0), \qquad \forall ~v \in C_c^\infty(\mathbb{R}^N).
	\]
	The system \eqref{system1.2} is reduced to
	\begin{equation}
		\left\{
		\begin{aligned}
			-\Delta v_1 + \mu_1 v_1 &=  |v_1|^{2a -2} v_1 + \theta p |v_2|^q|v_1|^{p-2} v_1 + \varepsilon v_2, \qquad in \,\, \mathbb{R}^N, \\
			-\Delta v_2 + \mu_2 v_2 &=|v_2|^{2b -2} v_2 + \theta q |v_1|^p|v_2|^{q-2} v_2 + \varepsilon v_1 , \qquad in \,\, \mathbb{R}^N.
		\end{aligned}
		\right.
		\label{system1.3}
	\end{equation}
	\par Two methods are available for handling system \eqref{system1.3}: one is to make the frequencies $\mu_1$ and $\mu_2$ fixed, the other is to fix the $L^2$-norms $|v_1|_2$ and $|v_2|_2$. 
	 The first method is referred to \cite{lin2005ground,Ambrosetti2007,Bartsch2010,Li2016,Luo2019}. In this paper, the second approach will be employed to obtain the existence results for system \eqref{system1.1}. Since the publication of Jeanjean's pioneering article \cite{Jeanjean1997}, numerous scholars have been engaged in researching normalized solutions during the past twenty years \cite{soave2020normalized,bartsch2023existence,Bartsch2019,Bartsch2016,Bartsch2018}. Notably, the Sobolev subcritical case has been the focus of substantial academic interest and Soave carried out fundamental work for the NLS equation in \cite{Soave2020}. After that, Bartsch et al. generalized the result to the Sobolev critical Schr{\"o}dinger system \cite{bartsch2023existence}. They proved the existence and non-existence of normalized ground state for system \eqref{system1.3} with $N=3\,\, \text{or} \,\, 4$,  $p,q>1 $, $p+q\in(2,2^*)$, $a=b=\frac{2^*}{2}=\frac{N}{N-2}$, $\varepsilon=0$.
	
	For the case that linear coupling occurs in the Schr{\"o}dinger system, several new challenges have been identified in $L^2$-mass fixed problem. To our current understanding, Yun and Zhang\cite{Yun2024} first investigated $\varepsilon=\varepsilon(x) \in L^\infty(\mathbb{R}^N)$ for the doubly coupled nonlinear Schr{\"o}dinger system. Using variational techniques, they demonstrated both the existence and symmetry of the solutions. Regarding the situation of the pure linear couplings for Schr{\"o}dinger systems, see \cite{chen2021normalized} for the mass subcritical case, \cite{chen2021normalizedsolutionsnonlinearschrodinger} for the mass supercritical case and \cite{Chen2023} for the combined case. In \cite{deng2025normalized}, Deng and Zhang investigated the system \eqref{system1.3} with $N=3\,\, \text{or} \,\, 4$, $a=b=\frac{2^*}{2}=\frac{N}{N-2}$, $p,q>1$, $p+q \in (2,2^*] \setminus \{2+\frac{4}{N}\}$.
	They demonstrated when $\theta>0$ and $p+q\in (2,2+\frac{4}{N})$, the system \eqref{system1.3} possesses a  normalized solution as $0<\varepsilon<\varepsilon_0$ and the prescribed $L^2$-mass small enough. For $\theta>0$ and $p+q \in (2+\frac{4}{N},2^*)$, the system \eqref{system1.3} possesses a normalized solution as $0<\theta_*<\theta<\theta_{**}$ and $0<\varepsilon<\varepsilon_1$.
	
	Now let $H(x)=I_\omega(x)$, $\theta=0$ and $\varepsilon=0$, the system \eqref{system1.2} is reduced to the Choquard equation as follows
	\begin{equation}
		-\Delta v + \mu v = ( I_\omega * |v|^r ) |v|^{r-2}v .
		\label{equation1.4}
	\end{equation}
	
	\par In \cite{MOROZ2013153}, Moroz and Schaftingen established that equation \eqref{equation1.4} admits a nontrivial solution exactly when $\frac{N+\omega}{N}<r<\frac{N+\omega}{N-2}$. When considering single Choquard equations, extensive results have been established(see \cite{battaglia2018groundstates,MOROZ2013153,moroz2015existence,Van2018} for details). However, research on Choquard systems has received comparatively limited attention. The reader may consult \cite{zhang2023normalizedsolutionscriticalchoquard,Li2025}. In \cite{zhang2023normalizedsolutionscriticalchoquard}, Zhang et al. considered the system \eqref{system1.1} with $\varepsilon=0$. They established the existence results for the critical Choquard system with pure nonlinear couplings. Significantly, Moroz proposed a nonlocal Brezis-Kato type regularity estimate in \cite{moroz2015existence}, which is crucial for establishing compactness.
	
	\par Our study is devoted to investigating existence results for system \eqref{system1.1}. In this article, suppose that $N=3\,\, \text{or} \,\, 4$, $\alpha_1,\alpha_2 > 0$, $\theta > 0$ and $2_{\omega,*} :=\frac{N+\omega}{N} <p,q<2_\omega^*:=\frac{N+\omega}{N-2}$.

\par Define $H:=H^1_r(\mathbb{R}^N) \times H^1_r(\mathbb{R}^N)$, where $H^1_r(\mathbb{R}^N)=\{v \in H^1(\mathbb{R}^N):v(x)=v(|x|)\}$,  the energy function $\mathcal{J}_\theta : H \to \mathbb{R}$ is defined as follows
\begin{equation*}
	\begin{split}
	\mathcal{J}_\theta(v_1,v_2)=&\frac{1}{2}(|\nabla v_1|_2^2+|\nabla v_2|_2^2)-\frac{1}{22_\omega^*}\int_{\mathbb{R}^N}\left((I_\omega*|v_1|^{2_\omega^*})|v_1|^{2_\omega^*}+(I_\omega*|v_2|^{2_\omega^*})|v_2|^{2_\omega^*}\right)\\
	&-\theta \int_{\mathbb{R}^N}(I_\omega*|v_1|^p)|v_2|^q-\varepsilon \int_{\mathbb{R}^N}v_1v_2.
    \end{split}
\end{equation*}
Denote the $L^2-trous$
\[ \mathcal{S}(\alpha_1,\alpha_2):= \{(v_1,v_2)\in H:\int_{\mathbb{R}^N} v_1^2 = \alpha_1^2 ,~ \int_{\mathbb{R}^N} v_2^2 = \alpha_2^2 \}. \]
Then the critical points of $\mathcal{J}_\theta(v_1,v_2)$ restricted to  $\mathcal{S}(\alpha_1,\alpha_2)$ coincide with the weak solutions of system $\eqref{system1.1} $, where $\mu_1, \mu_2 \in \mathbb{R}$ are the corresponding Lagrangian multipliers. 

\par Next we introduce the Pohozaev identity 
\begin{equation}
	\begin{split}
		P_\theta(v_1,v_2)=&|\nabla v_1|_2^2+|\nabla v_2|_2^2-\int_{\mathbb{R}^N}\left((I_\omega*|v_1|^{2_\omega^*})|v_1|^{2_\omega^*}+(I_\omega*|v_2|^{2_\omega^*})|v_2|^{2_\omega^*}\right)\\
	&-\theta(\gamma_p+\gamma_q)\int_{\mathbb{R}^N}(I_\omega*|v_1|^p)|v_2|^q,
	\end{split}
\end{equation}
where $\gamma_p:=\frac{Np-N-\omega}{2}$. Let $\mathcal{P}_\theta(\alpha_1,\alpha_2):=\{(v_1,v_2)\in H: P_\theta(v_1,v_2)=0, (v_1,v_2)\in \mathcal{S}(\alpha_1,\alpha_2)\}$, we consider the following minimizing problem
\[ m_\theta(\alpha_1,\alpha_2):= \inf_{ \mathcal{P}_\theta(\alpha_1,\alpha_2)}\mathcal{J}_\theta.
\]
A solution $(v_1, v_2)$ to system $\eqref{system1.1}$ is referred to as a normalized ground state provided that it minimizes $m_\theta(\alpha_1, \alpha_2)$.

\par The principal results established in this work are stated in the next two theorems. For the $L^2$-subcritical case $p+q<\frac{2N+2\omega+4}{N}$, we present the first theorem.
\begin{thm}
	Let $N=3\,\, \text{or} \,\, 4$, $2_{\omega,*} =\frac{N+\omega}{N} <p,q<2_\omega^*=\frac{N+\omega}{N-2}$ and  $p+q<\frac{2N+2\omega+4}{N} $, there exists $\theta_0>0$, $\varepsilon_*=\varepsilon_*(\theta_0)>0$ such that system $\eqref{system1.1} $ possesses a normalized ground state $(v_1,v_2)$ for every $0<\theta<\theta_0,\;0<\varepsilon<\varepsilon_*$. Furthermore,
	\begin{enumerate}[label=(\roman*)]
		\item $v_i \in H_r^1(\mathbb{R}^N) $ is positive and the corresponding Lagrangian multiplier $\mu_i$ is positive, $i=1,2$.
		\item  $(v_1,v_2)$ is a local minimizer of $\mathcal{J}_\theta$ restricted to $\mathcal{S}(\alpha_1,\alpha_2)$.
	\end{enumerate}
	\label{Theorem1.1}
\end{thm}
The compactness argument plays a pivotal role in establishing Theorem \ref{Theorem1.1}. It's essential to overcome the difficulties arising from both the nonlocal terms and the linear couplings. For the nonlocal terms, we need to improve the integrability and the regularity conditions. For the linear couplings, we require the Brezis-Lieb type identity and the Liouville type lemma. The detailed proof is provided in Lemma \ref{Lemma3.7紧性}.

\par Then for the $L^2$-supercritical case $p+q>\frac{2N+2\omega+4}{N}$, we now present the second theorem.
\begin{thm}
	Let $N=3\,\, \text{or} \,\, 4$, $2_{\omega,*} =\frac{N+\omega}{N} <p,q<2_\omega^*=\frac{N+\omega}{N-2} \; and \; p+q>\frac{2N+2\omega+4}{N} $, 
	there exists $\theta_*>0,\; \overline{\varepsilon}>0$ such that system $\eqref{system1.1} $ possesses a normalized ground state $(v_1,v_2)$ for every $\theta>\theta_*,\;0<\varepsilon<\overline{\varepsilon}$. Furthermore,
	\begin{enumerate}[label=(\roman*)]
		\item $v_i \in H_r^1(\mathbb{R}^N) $ is positive and the corresponding Lagrangian multiplier $\mu_i$ is positive, $i=1,2$.
		\item  $(v_1,v_2)$ is  a minimax type critical point of $\mathcal{J}_\theta$ restricted to $\mathcal{S}(\alpha_1,\alpha_2)$.
	\end{enumerate}
	\label{theorem1.2}
\end{thm}	
	\begin{rmk}
    In Theorem \ref{theorem1.2}, we remove the upper-bound restriction on $\theta$. This methodology remains applicable to \cite[Theorem 1.2]{deng2025normalized}.
    \end{rmk}
This article proceeds as detailed below. Section 2 is devoted to providing the essential technical preliminaries. In section 3, we first enhance the solution's regularity, then we establish the existence results for both the $L^2$-subcritical and supercritical regimes when $\theta>0 $ and $ \varepsilon>0$.

	\section{PRELIMINARIES}
	
	\par For completeness, the proposition below states the classical Hardy–Littlewood–Sobolev inequality.
\begin{prop}\cite[theorem 4.3]{lieb2001analysis}
	Let $p \in (1,\frac{N}{\omega})$ , $ u \in L^p(\mathbb{R}^N)$ , then it holds that $I_\omega*u \in L^{\frac{Np}{N-\omega p}}(\mathbb{R}^N)$ and 
	\begin{equation}
		\int_{\mathbb{R}^N}|I_\omega*u|^{\frac{Np}{N-\omega p}} \leq C \left(\int_{\mathbb{R}^N} |u|^p\right)^{\frac{N}{N-\omega p}} ,                      \label{HLS}
	\end{equation}
	here C only depends on N, $\omega$ and p.
\end{prop}	
Applying \eqref{HLS}, together with Hölder’s and Gagliardo-Nirenberg inequalities, for $2_{\omega,*}  <p<2_\omega^*$, we have
\[ \int_{\mathbb{R}^N}(I_\omega *|v|^p)|v|^p \leq C \left(\int_{\mathbb{R}^N} |v|^{\frac{2Np}{N+\omega}}\right)^{\frac{N+\omega}{N}} \leq C' |\nabla v|_2^{2 \gamma_p}|v|_2^{2p-2\gamma_p}.
\]
Then we consider the semi-group identity of the Riesz potential $ I_\omega=I_{\frac{\omega}{2}}*I_{\frac{\omega}{2}}$ \cite[theorem 5.9]{lieb2001analysis}. For the vector-valued case (see (3.3) in \cite{ghimenti2016nodal}), it follows that
\begin{align*}
    \int_{\mathbb{R}^N}(I_\omega *|v_1|^p)|v_2|^q & \leq C \left(\int_{\mathbb{R}^N}(I_\omega *|v_1|^p)|v_1|^p\right)^{\frac{1}{2}} \left(\int_{\mathbb{R}^N}(I_\omega *|v_2|^q)|v_2|^q \right)^{\frac{1}{2}} \\
    & \leq C |\nabla v_1|_2^{\gamma_p} |\nabla v_2|_2^{\gamma_q} |v_1|_2^{p-\gamma_p} |v_2|_2^{q-\gamma_q} \\
    & \leq C' \left(|\nabla v_1|_2^2+|\nabla v_2|_2^2\right)^{\frac{\gamma_p+\gamma_q}{2}} \left(|v_1|_2^2+|v_2|_2^2\right)^{\frac{p+q-\gamma_p-\gamma_q}{2}}.
\end{align*}
The last inequality is provided by Young's inequality.

\par Let \[C(N,p,q)^{-1} := \inf_{(v_1,v_2)\in H^1(\mathbb{R}^N)\setminus \{0\}} \frac{\left(|\nabla v_1|_2^2+|\nabla v_2|_2^2\right)^{\frac{\gamma_p+\gamma_q}{2}} \left(|v_1|_2^2+|v_2|_2^2\right)^{\frac{p+q-\gamma_p-\gamma_q}{2}}}{\int_{\mathbb{R}^N}(I_\omega *|v_1|^p)|v_2|^q}.\]
We have 
\begin{equation}
	\int_{\mathbb{R}^N}(I_\omega *|v_1|^p)|v_2|^q \leq C(N,p,q) \left(|\nabla v_1|_2^2+|\nabla v_2|_2^2\right)^{\frac{\gamma_p+\gamma_q}{2}} \left(|v_1|_2^2+|v_2|_2^2\right)^{\frac{p+q-\gamma_p-\gamma_q}{2}}.
	 \label{GN}
\end{equation}

\par Next, we introduce another important constant
\begin{equation}
	S_H:=\inf_{v\in D^{1,2}(\mathbb{R}^N) \setminus \{0\}}\frac{|\nabla v|_2^2}{(\int_{\mathbb{R}^N}(I_\omega*|v|^{2_\omega^*})|v|^{2_\omega^*})^{\frac{1}{2_\omega^*}}}.
	\label{SH}
\end{equation}

We now introduce the $L^2$-invariant scaling.
Define $t\diamond u(x):=t^{\frac{N}{2}}u(tx)$ for $t>0$ and $t\diamond(v_1,v_2):=(t \diamond v_1, t \diamond v_2)$. Notice that $|t \diamond u|_2^2=|u|^2_2$. Then for $(v_1,v_2) \in \mathcal{S}(\alpha_1,\alpha_2)$ , consider the map $\Psi_{(v_1,v_2)}^\theta : (0,+\infty) \to \mathbb{R}$,
\begin{equation*}
	\begin{split}
		\Psi_{(v_1,v_2)}^\theta(t):=&\mathcal{J}_\theta(t\diamond v_1,t\diamond v_2)=\frac{t^2}{2}(|\nabla v_1|_2^2+|\nabla v_2|_2^2)-\frac{t^{22_\omega^*}}{22_\omega^*}\int_{\mathbb{R}^N}((I_\omega*|v_1|^{2_\omega^*})|v_1|^{2_\omega^*}+\\
		&(I_\omega*|v_2|^{2_\omega^*})|v_2|^{2_\omega^*})-\theta t^{\gamma_p+\gamma_q}\int_{\mathbb{R}^N}(I_\omega*|v_1|^p)|v_2|^q-\varepsilon \int_{\mathbb{R}^N}v_1v_2.
	\end{split}
\end{equation*}
By direct computation,  $(\Psi_{(v_1,v_2)}^\theta)'(t)=\frac{1}{t}P_\theta(t\diamond v_1,t\diamond v_2)$. Hence, 
\[\mathcal{P}_\theta(\alpha_1,\alpha_2)=\{(v_1,v_2) \in \mathcal{S}(\alpha_1,\alpha_2) :(\Psi_{(v_1,v_2)}^\theta)'(1)=0\}.\]
We consider the following decomposition of $\mathcal{P}_\theta(\alpha_1,\alpha_2)$:
\begin{gather*}
	\mathcal{P}_\theta^+(\alpha_1,\alpha_2)=\{(v_1,v_2) \in \mathcal{S}(\alpha_1,\alpha_2):(\Psi_{(v_1,v_2)}^\theta)''(1)>0\},\\
	\mathcal{P}_\theta^0(\alpha_1,\alpha_2)=\{(v_1,v_2) \in \mathcal{S}(\alpha_1,\alpha_2):(\Psi_{(v_1,v_2)}^\theta)''(1)=0\},\\
	\mathcal{P}_\theta^-(\alpha_1,\alpha_2)=\{(v_1,v_2) \in \mathcal{S}(\alpha_1,\alpha_2):(\Psi_{(v_1,v_2)}^\theta)''(1)<0\}.
\end{gather*}

    \section{Existence results when $\theta>0 $ and $ \varepsilon>0$ }
    
    This section is devoted to proving the existence results when $\theta>0 $ and $ \varepsilon>0$. Before this, we need regularity conditions for the nonlocal terms.
    \begin{lem}\cite[Lemma 3.2]{moroz2015existence}
    	Suppose $N \geq 2$, $0<\omega,s<2$, if $P,Q \in L^{\frac{2N}{\omega+2}}(\mathbb{R}^N)+L^{\frac{2N}{\omega}}(\mathbb{R}^N)$ and $\frac{\omega}{N}<s<2-\frac{\omega}{N}$, then for each $\epsilon>0$, one can find $C_{\epsilon,s}>0$ satisfying
    	\[ \int_{\mathbb{R}^N}[I_\omega*(P|\varphi|^s)]Q|\varphi|^{2-s} \leq \epsilon \int_{\mathbb{R}^N} |\nabla \varphi|^2 + C_{\epsilon,s}\int_{\mathbb{R}^N} |\varphi|^2,            \qquad \forall \, \varphi \in H^1(\mathbb{R}^N).
    	\] \label{Lemma3.1 Moroz}
    \end{lem}

	\begin{lem}
		If $ N \geq 2 ,\,  P_{ij} ,\, Q_{ij} \in L^{\frac{2N}{\omega+2}}(\mathbb{R}^N)+L^{\frac{2N}{\omega}}(\mathbb{R}^N)$, $i,j\in \{1,2\}$, $(v_1,v_2) \in H$ solves the following system
	\begin{equation}
		\left\{
		\begin{aligned}
		-\Delta v_1 + \mu_1 v_1 &= ( I_\omega * P_{11}v_1) Q_{11} + ( I_\omega * P_{12}v_2)Q_{12}+ \varepsilon v_2, \qquad \text{in} \,\, \mathbb{R}^N, \\
		-\Delta v_2 + \mu_2 v_2 &= ( I_\omega * P_{21}v_2 ) Q_{21} + ( I_\omega * P_{22}v_1)Q_{22} + \varepsilon v_1 , \qquad \text{in} \,\, \mathbb{R}^N ,
		\end{aligned}
		\right.
		\label{system3.1}
	\end{equation}
		then for any $p \in [2,\frac{N}{\omega}\cdot\frac{2N}{N-2})$, $v_1,v_2 \in L^p(\mathbb{R}^N)$.
		\label{Lemma3.2提升可积性}
	\end{lem}
	\begin{proof}
		Applying Lemma \ref{Lemma3.1 Moroz} with $s=1$, for $\forall~ \xi \in H^1(\mathbb{R}^N)$, there exist $\delta >0$ such that 
		\[\int_{\mathbb{R}^N} (I_\omega*|P_{i1}\xi|)|Q_{i1}\xi| \leq \frac{1}{4}|\nabla \xi|_2^2+\frac{\delta}{4}|\xi|_2^2, \qquad  i=1,2.
		\]
		By Lemma \ref{Lemma3.1 Moroz} and (3.3) in \cite{ghimenti2016nodal}, for $(\xi,\zeta)\in H$, $i=1,2$, we obtain
		\begin{align*}
			\int_{\mathbb{R}^N} (I_\omega*|P_{i2}\xi|)|Q_{i2}\zeta|&\leq (\int_{\mathbb{R}^N} (I_\omega*|P_{i2}\xi|)|P_{i2}\xi| )^{\frac{1}{2}} (\int_{\mathbb{R}^N} (I_\omega*|Q_{i2}\zeta|)|Q_{i2}\zeta| )^{\frac{1}{2}}\\
			&\leq (\frac{1}{4}|\nabla \xi|_2^2+\frac{\delta}{4}|\xi|_2^2)^{\frac{1}{2}} (\frac{1}{4}|\nabla \zeta|_2^2+\frac{\delta}{4}|\zeta|_2^2)^{\frac{1}{2}}\\
			&\leq \frac{1}{8}\int_{\mathbb{R}^N}(|\nabla \xi|^2+|\nabla \zeta|^2)+\frac{\delta}{8}\int_{\mathbb{R}^N}(|\xi|^2+| \zeta|^2).
		\end{align*}
		
		\par Let $(P_{ij,k})_{k \in \mathbb{N}}$ and $(Q_{ij,k})_{k \in \mathbb{N}}$ be sequences in $L^{\frac{2N}{\omega}}(\mathbb{R}^N)$ satisfying $|P_{ij,k}| \leq |P_{ij}|$, $|Q_{ij,k}| \leq |Q_{ij}|$ and $P_{ij,k} \to P_{ij}$ and $Q_{ij,k} \to Q_{ij}$ almost everywhere in $\mathbb{R}^N$, $i,j=1,2$. For $(\xi_1,\xi_2),(\zeta_1,\zeta_2) \in H$, define the form $a_k:H\times H\to \mathbb{R}$ as follows:
	    \begin{align*}
	    	a_k((\xi_1,\xi_2),(\zeta_1,\zeta_2))&=\int_{\mathbb{R}^N} (\nabla \xi_1 \nabla \zeta_1+\delta \xi_1 \zeta_1)-\int_{\mathbb{R}^N}(I_\omega*|P_{11,k}\xi_1|)|Q_{11,k}\zeta_1|\\
	    	&-\int_{\mathbb{R}^N}(I_\omega*|P_{12,k}\xi_2|)|Q_{12,k}\zeta_1|
	    	+\int_{\mathbb{R}^N} (\nabla \xi_2 \nabla \zeta_2+\delta \xi_2 \zeta_2)\\
	    	&-\int_{\mathbb{R}^N}(I_\omega*|P_{21,k}\xi_2|)|Q_{21,k}\zeta_2|-\int_{\mathbb{R}^N}(I_\omega*|P_{22,k}\xi_1|)|Q_{22,k}\zeta_2|.
	    \end{align*}
	    
	    \par Given the bilinearity and coercivity of $a_k$, the Lax–Milgram theorem guarantees the existence and uniqueness of a solution $(v_{1,k},v_{2,k}) \in H$ to the following system.
	    {\small
	    \begin{equation}
	    	\left\{
	    	\begin{aligned}
	    		-\Delta v_{1,k}+ \delta v_{1,k} &= ( I_\omega * P_{11,k}v_{1,k}) Q_{11,k} + ( I_\omega * P_{12,k}v_{2,k})Q_{12,k}+ \varepsilon v_2+(\delta-\mu_1)v_1, \\
	    		-\Delta v_{2,k} + \delta v_{2,k} &= ( I_\omega * P_{21,k}v_{2,k} ) Q_{21,k} + ( I_\omega * P_{22,k}v_{1,k})Q_{22,k} + \varepsilon v_1+(\delta-\mu_2)v_2 ,
	    	\end{aligned}
	    	\right.
	    	\label{system3.2} 
	    \end{equation}
	    }
	    where $(v_1,v_2) \in H$ is the solution to system \eqref{system3.1}. It can be verified that  $(v_{1,k},v_{2,k}) \rightharpoonup (v_1,v_2)$ in $H$.
	    
	    Let $w_k=|v_{1,k}|+|v_{2,k}|$, $P_k=\sum\limits_{i,j=1}^{2}|P_{ij,k}|$ and $Q_k=\sum\limits_{i,j=1}^{2}|Q_{ij,k}|$. For $\lambda>0$, we introduce the truncation $w_{k,\lambda}: \mathbb{R}^N \to \mathbb{R}$ defined by
	    \[
	    w_{k,\lambda}(x)= 
	    \begin{cases}
	    	\lambda, & \text{if } w_k(x)\geq \lambda, \\
	    	w_{k,\lambda}(x), & \text{if } w_k(x)< \lambda.
	    \end{cases}
	    \]
	   
	   \par By \cite[Lemma A]{kato1972schrodinger}, $-sgn(v_{i,k})\Delta v_{i,k} \geq- \Delta |v_{i,k}|$,  in the sense of distribution, $i=1,2$.
	    Testing \eqref{system3.2} with $(sgn(v_{1,k}),sgn(v_{2,k}))$, we have
	    \begin{equation}
	    	-\Delta w_k+\delta w_k \leq (I_\omega*(P_kw_k))Q_k+C(|v_1|+|v_2|).
	    	\label{3.3}
	    \end{equation}
	    \par Now we test \eqref{3.3} with $w_{k,\lambda}^{p-1} \in H^1(\mathbb{R}^N)$:
	    {\small
	    \begin{align*}
	    	\int_{\mathbb{R}^N} \left( \frac{4(p-1)}{p^2}|\nabla (w_{k,\lambda})^\frac{p}{2}|^2+\delta |w_{k,\lambda}|^p \right) & \leq \int_{\mathbb{R}^N} \left( (p-1)|w_{k,\lambda}|^{p-2}|\nabla w_{k,\lambda}|^2+\delta |w_{k,\lambda}|^{p-1}w_k \right)\\
	    	&\leq \int_{\mathbb{R}^N} \left( (I_\omega*(P_kw_k))Q_k|w_{k,\lambda}|^{p-1}+C(|v_1|+|v_2|)|w_{k,\lambda}|^{p-1} \right).
	    \end{align*}
	    }
	     
	     \par Let $P=\sum\limits_{i,j=1}^{2}|P_{ij}|$ and $Q=\sum\limits_{i,j=1}^{2}|Q_{ij}|$. If $p<\frac{2N}{\omega}$, applying Lemma \ref{Lemma3.1 Moroz} with $s=\frac{2}{p}$, we have
	    \begin{align*}
	    	\int_{\mathbb{R}^N} (I_\omega*(P_kw_{k,\lambda}))Q_k|w_{k,\lambda}|^{p-1}&\leq \int_{\mathbb{R}^N}  (I_\omega*(Pw_{k,\lambda}))Q|w_{k,\lambda}|^{p-1}\\
	    	& \leq \frac{2(p-1)}{p^2} \int_{\mathbb{R}^N} |\nabla (w_{k,\lambda})^{\frac{p}{2}}|^2+C' |w_{k,\lambda}|_p^p.
	    \end{align*}
	    Thus, we obtain
	    \[ \frac{2(p-1)}{p^2} \int_{\mathbb{R}^N} |\nabla (w_{k,\lambda})^{\frac{p}{2}}|^2 \leq C''\int_{\mathbb{R}^N}(|w_k|^p+|v_1|^p+|v_2|^p)+\int_{A_{k,\lambda}} (I_\omega*(P_kw_k))Q_k|w_k|^{p-1}, \]
	    where $A_{k,\lambda}=\{x \in \mathbb{R}^N: w_k(x)>\lambda\}$. By Proposition \ref{HLS}, we get
	    \[\int_{A_{k,\lambda}} (I_\omega*(P_kw_k))Q_k|w_k|^{p-1} \leq C'''\left(\int_{A_{k,\lambda}}|P_kw_k|^t \right)^{\frac{1}{t}}
	    \left(\int_{\mathbb{R}^N}(Q_k|w_k|^{p-1})^r\right)^{\frac{1}{r}},\]
	    where $\frac{1}{t}=\frac{\omega}{2N}+\frac{1}{p}$ and $\frac{1}{r}=\frac{\omega}{2N}+1-\frac{1}{p}$. If $w_k \in L^p(\mathbb{R}^N)$, then $P_kw_k \in L^t(\mathbb{R}^N)$ and $Q_k|w_k|^{p-1} \in L^r(\mathbb{R}^N)$ by Hölder's inequality. Using Lebesgue’s dominated convergence theorem, one derives that
	    \[ \lim\limits_{\lambda \to +\infty} \int_{A_{k,\lambda}} (I_\omega*(P_kw_k))Q_k|w_k|^{p-1}=0.
	    \]
	    
	    \par Therefore, $\limsup\limits_{k \to +\infty} |w_k|_{\frac{Np}{N-2}} \leq \tilde{C} \limsup\limits_{k \to +\infty} |w_k|_p$. Finally, by iterating over p finitely many times, one covers the range $p \in [2,\frac{N}{\omega}\cdot\frac{2N}{N-2})$.
	\end{proof}

	\begin{lem}
		Suppose $(v_1,v_2) \in H$ solves the following system
		 \begin{equation*}
		 	\left\{
		 	\begin{aligned}
		 		-\Delta v_1 + \mu_1 v_1 &= ( I_\omega * |v_1|^{2_\omega^*} ) |v_1|^{2_\omega^* -2} v_1 + \theta p( I_\omega * |v_2|^q)|v_1|^{p-2} v_1 + \varepsilon v_2, \qquad in \,\, \mathbb{R}^N, \\
		 		-\Delta v_2 + \mu_2 v_2 &= ( I_\omega * |v_2|^{2_\omega^*} ) |v_2|^{2_\omega^* -2} v_2 + \theta q( I_\omega * |v_1|^p)|v_2|^{q-2} v_2 + \varepsilon v_1,  \qquad in \,\, \mathbb{R}^N,
		 	\end{aligned}
		 	\right.
		 \end{equation*}
		 then $(v_1,v_2)$ is a smooth solution.
		 \label{smooth}
	\end{lem}
	\begin{proof}
		Applying Lemma \ref{Lemma3.2提升可积性} with $P_{11}=Q_{11}=|v_1|^{2_\omega^* -2} v_1$, $P_{12}=\frac{1}{q}Q_{22}=|v_2|^{q-2}v_2$, $P_{21}=Q_{21}= |v_2|^{2_\omega^* -2} v_2$ and
		$P_{22}=\frac{1}{p}Q_{12}=|v_1|^{p-2}v_1$. Thus, $P_{ij} ,\, Q_{ij} \in L^{\frac{2N}{\omega+2}}(\mathbb{R}^N)+L^{\frac{2N}{\omega}}(\mathbb{R}^N) ,\, i,j=1,2$.
		Hence, $v_1,v_2 \in L^p(\mathbb{R}^N)$, for any $p \in [2,\frac{N}{\omega}\cdot\frac{2N}{N-2})$. Then we have $|v_2|^q,|v_1|^{2_\omega^*} \in L^{\frac{N}{\omega}}(\mathbb{R}^N)$, which implies that $I_\omega*|v_2|^q$, $I_\omega*|v_1|^{2_\omega^*} \in L^{\infty}(\mathbb{R}^N)$. Therefore,
		\begin{equation}
			-\Delta v_1 +\mu_1 v_1 \leq C (|v_1|^{\frac{\omega}{N}}+|v_1|^{\frac{\omega+2}{N-2}})+\varepsilon v_2.
			\label{3.4}
		\end{equation}
		Similarly, we obtain
		\begin{equation}
		-\Delta v_2 +\mu_2 v_2 \leq C (|v_2|^{\frac{\omega}{N}}+|v_2|^{\frac{\omega+2}{N-2}})+\varepsilon v_1.
		\label{3.5}
		\end{equation}
		By summing \eqref{3.4} and \eqref{3.5}, it follows from \cite[Lemma A]{kato1972schrodinger} that
		\[ -\Delta (|v_1|+|v_2|)+(|v_1|+|v_2|)  \leq C \left((|v_1|+|v_2|)^{\frac{\omega}{N}}+(|v_1|+|v_2|)^{\frac{\omega+2}{N-2}}\right).
		\]
		
		\par Similar to the proof of Theorem 2 in \cite{moroz2015existence}, by the classical bootstrap method for the Sobolev subcritical local problems in bounded domains, one derives that $|v_1|+|v_2| \in W_{loc}^{2,q}(\mathbb{R}^N)$ for all $q \geq 1$.
		Then by \cite[Theorem 4.1]{Han2011} and \eqref{3.4}, we get $v_1 \in W_{loc}^{2,q}(\mathbb{R}^N)$ for all $q \geq 1$.
	\end{proof}
	
	Having established the regularity conditions, we now proceed to review the following Liouville type lemma.
	\begin{lem}
		\cite[Lemma A.2]{ikoma2014compactness} Assume $q \in (0,+ \infty)$ for $N=1,2$, while $q \in (0,\frac{N}{N-2}]$ for $N \geq 3$. Suppose $v \in L^q(\mathbb{R}^N)$ is smooth, nonnegative, and satisfies $- \Delta v \geq 0$ in $\mathbb{R}^N$. It follows that $v \equiv 0$.
		\label{Liouville}
	\end{lem}

\subsection{The $L^2$-subcritical case $p+q<\frac{2N+2\omega+4}{N}$}\leavevmode\par
	Denote\\ $\theta_1:=\frac{S_H^{\frac{(2-\gamma_p-\gamma_q)2_\omega^*}{22_\omega^*-2}}}{C(N,p,q)(\alpha_1^2+\alpha_2^2)^{\frac{p+q-\gamma_p-\gamma_q}{2}}} \cdot \frac{(22_\omega^*-2)(2-\gamma_p-\gamma_q)^\frac{2-\gamma_p-\gamma_q}{22_\omega^*-2}}{(22_\omega^*-\gamma_p-\gamma_q)^{\frac{22_\omega^*-\gamma_p-\gamma_q}{22_\omega^*-2}}}$ , $\theta_2:= \frac{1}{\gamma_p+\gamma_q}\theta_1$ , $\theta_3:=\frac{22_\omega^*+2-\gamma_p-\gamma_q}{42_\omega^*}\theta_1$.\\
	Let \[\theta_0:=min\{\theta_2,\theta_3\}.\]
	
	\par For $s>0$, we define\\
	\[h(s):= \frac{1}{2}s^2-As^{\gamma_p+\gamma_q}-Bs^{22_\omega^*}-\varepsilon\alpha_1\alpha_2  \;\;\text{and}\;\; g(s):=s^{2-(\gamma_p+\gamma_q)}-2\cdot2_\omega^*Bs^{22\omega^*-(\gamma_p+\gamma_q)},     \]
    where $A=\theta C(N,p,q)(\alpha_1^2+\alpha_2^2)^{\frac{p+q-(\gamma_p+\gamma_q)}{2}} \,,\; B=\frac{1}{2\cdot2_\omega^*S_H^{2_\omega^*}}$.
    
    \par Then $h'(s)=s^{\gamma_p+\gamma_q-1}\left( g(s)-(\gamma_p+\gamma_q)A \right)$ and $J_\theta(v_1,v_2) \geq h \left( (|\nabla v_1|_2^2+|\nabla v_2|_2^2)^{\frac{1}{2}} \right)$.\\
    The last inequality is provided by \eqref{GN} and \eqref{SH}.
   
   \par  Let
    \begin{equation}
     s_*:=(\frac{2-\gamma_p-\gamma_q}{22_\omega^*-\gamma_p-\gamma_q}S_H^{2_\omega^*})^{\frac{1}{22_\omega^*-2}} ,\; \varepsilon_*:=min\{\frac{\frac{1}{2}s_*^2-A_0s_*^{\gamma_p+\gamma_q}-Bs_*^{22_\omega^*}}{\alpha_1 \alpha_2},\, \frac{2_\omega^*-1}{\alpha_1\alpha_222_\omega^*}S_H^{\frac{2_\omega^*}{2_\omega^*-1}}\}, \label{s_*}
  \end{equation}
	where $A_0=\theta_0C(N,p,q)(\alpha_1^2+\alpha_2^2)^{\frac{p+q-\gamma_p-\gamma_q}{2}} $.
	
	\par By direct calculation we know $g(s)$ increases on $(0,s_*)$ and decreases on $(s_*,+\infty)$. 
	Assume $0<\theta<\theta_0$, $0<\varepsilon<\varepsilon_*$, then $g(s_*)>(\gamma_p+\gamma_q)A$,
	$h(s_*)>\varepsilon \alpha_1\alpha_2>0$. Therefore, $h(s)$ has two critical points $0<s_1<s_*<s_2$ with
	\[h(s_1)=\min\limits_{0<s<s_*} h(s)<0,\quad h(s_2)=\max_{s>0}h(s)>\varepsilon\alpha_1\alpha_2.
	\]
	
	\par Moreover, there exist $T_1>T_0>0$ satisfying $h(T_0)=h(T_1)=\varepsilon \alpha_1 \alpha_2$ and $h(s)>\varepsilon \alpha_1 \alpha_2$ if and only if $s \in (T_0,T_1)$.

	\begin{lem}
		Assume $0<\theta<\theta_0$, $0<\varepsilon<\varepsilon_*$, then for any $(v_1,v_2) \in \mathcal{S}(\alpha_1,\alpha_2)$, $\Psi_{(v_1,v_2)}^\theta(t)$ admits precisely two critical points $s_\theta(v_1,v_2)<t_\theta(v_1,v_2)$. Moreover,
\begin{enumerate}[label=(\roman*)]
			\item $\mathcal{P}_\theta^0(\alpha_1,\alpha_2)=\emptyset$ and $\mathcal{P}_\theta(\alpha_1,\alpha_2)$ is a submanifold of $H$.
			\item $s_\theta(v_1,v_2)\diamond (v_1,v_2) \in \mathcal{P}_\theta^+(\alpha_1,\alpha_2)$, $t_\theta(v_1,v_2)\diamond (v_1,v_2) \in \mathcal{P}_\theta^-(\alpha_1,\alpha_2)$ and $t \diamond (v_1,v_2) \in \mathcal{P}_\theta(\alpha_1,\alpha_2)$ exactly when $t=s_\theta(v_1,v_2) $ or $ t_\theta(v_1,v_2)$.
			\item $\Psi_{(v_1,v_2)}^\theta(s_\theta(v_1,v_2))=\min\{\Psi_{(v_1,v_2)}^\theta(t):t(\int_{\mathbb{R}^N}|\nabla v_1|_2^2+|\nabla v_2|_2^2)^{\frac{1}{2}}\leq T_0\}$.
			\item $\Psi_{(v_1,v_2)}^\theta(t)$ is strictly decreasing in $(t_\theta(v_1,v_2),+\infty)$ and  \[\Psi_{(v_1,v_2)}^\theta(t_\theta(v_1,v_2))=\max\limits_{t>0}\Psi_{(v_1,v_2)}^\theta(t)\].
			\item The functions $(v_1,v_2) \mapsto s_\theta(v_1,v_2)$ and $(v_1,v_2) \mapsto t_\theta(v_1,v_2)$ are of class $C^1$.
\end{enumerate}
\label{Lemma3.5变分结构}
	\end{lem}
	\begin{proof}
		$(\romannumeral1)$ To derive a contradiction, we assume that $(v_1,v_2) \in \mathcal{P}_\theta^0(\alpha_1,\alpha_2)$. Then 
		\begin{align*}
			(\Psi_{(v_1,v_2)}^\theta)'(1)=&\int_{\mathbb{R}^N}(|\nabla v_1|^2+|\nabla v_2|^2)- \int_{\mathbb{R}^N}\left((I_\omega*|v_1|^{2_\omega^*})|v_1|^{2_\omega^*}+(I_\omega*|v_2|^{2_\omega^*})|v_2|^{2_\omega^*}\right)\\
			&-\theta(\gamma_p+\gamma_q)\int_{\mathbb{R}^N}(I_\omega*|v_1|^p)|v_2|^q=0,
		\end{align*}
		\begin{align*}
			(\Psi_{(v_1,v_2)}^\theta)''(1)=&2\int_{\mathbb{R}^N}(|\nabla v_1|^2+|\nabla v_2|^2)- 2_\omega^*\int_{\mathbb{R}^N}\left((I_\omega*|v_1|^{2_\omega^*})|v_1|^{2_\omega^*}+(I_\omega*|v_2|^{2_\omega^*})|v_2|^{2_\omega^*}\right)\\
			&-\theta(\gamma_p+\gamma_q)^2 \int_{\mathbb{R}^N}(I_\omega*|v_1|^p)|v_2|^q=0.
		\end{align*}
		
		\par By \eqref{SH}, we have
		\begin{align*}
			|\nabla v_1|^2_2+|\nabla v_2|^2_2 &=\frac{22_\omega^*-\gamma_p-\gamma_q}{2-\gamma_p-\gamma_q}\int_{\mathbb{R}^N}\left((I_\omega*|v_1|^{2_\omega^*})|v_1|^{2_\omega^*}+(I_\omega*|v_2|^{2_\omega^*})|v_2|^{2_\omega^*}\right)\\
			&\leq \frac{22_\omega^*-\gamma_p-\gamma_q}{2-\gamma_p-\gamma_q}S^{-2_\omega^*}(|\nabla v_1|_2^{22_\omega^*}+|\nabla v_2|_2^{22_\omega^*})\\
			&\leq \frac{22_\omega^*-\gamma_p-\gamma_q}{2-\gamma_p-\gamma_q}S^{-2_\omega^*}(|\nabla v_1|_2^2+|\nabla v_2|_2^2)^{2_\omega^*}.
		\end{align*}
		Alternatively, one obtains from \eqref{GN} that
		{\small
		\begin{align*}
				|\nabla v_1|^2_2+|\nabla v_2|^2_2 &= \frac{22_\omega^*-\gamma_p-\gamma_q}{2(2_\omega^*-1)}\theta(\gamma_p+\gamma_q)\int_{\mathbb{R}^N}(I_\omega*|v_1|^p)|v_2|^q\\
				&\leq\frac{22_\omega^*-\gamma_p-\gamma_q}{2(2_\omega^*-1)}\theta(\gamma_p+\gamma_q)C(N,p,q)(\alpha_1^2+\alpha_2^2)^{\frac{p+q-\gamma_p-\gamma_q}{2}} (|\nabla v_1|_2^2+|\nabla v_2|_2^2)^{\frac{\gamma_p+\gamma_q}{2}},
		\end{align*}
	    }
		which is in contradiction with $\theta<\theta_0 \leq \theta_2$.
		
		\par We now proceed to show that $\mathcal{P}_\theta(\alpha_1,\alpha_2)$ lies in $H$ as a submanifold with codimension 3. The argument resembles the one used in \cite[Lemma 5.3]{soave2020normalized}. The details are provided below for better understanding. Note that
		\[\mathcal{P}_\theta(\alpha_1,\alpha_2)=\{(v_1,v_2)\in H: P_\theta(v_1,v_2),\, G_1(v_1)=0,\, G_2(v_2)=0\},
		\]
		where $G_1(v_1)=\alpha_1^2-|v_1|_2^2$,  $G_2(v_2)=\alpha_2^2-|v_2|_2^2$. 
		
		\par Therefore, we need to prove that $d(P_\theta,G_1,G_2): H \to \mathbb{R}^3$ is surjective, for any $(v_1,v_2) \in \mathcal{P}_\theta(\alpha_1,\alpha_2)$. If not, since $dG_1(v_1)$ and $dG_2(v_2)$ are independent, it holds that $dP_\theta(v_1,v_2)$ can be expressed as a linear combination of $dG_1(v_1)$, $dG_2(v_2)$.  There exists $\mu_1, \mu_2 \in \mathbb{R}$ such that $(v_1,v_2)$ solves
		\begin{equation*}
		\left\{
		\begin{aligned}
			-\Delta v_1 + \mu_1 v_1 &=2_\omega^* ( I_\omega * |v_1|^{2_\omega^*} ) |v_1|^{2_\omega^* -2} v_1 + \theta p\frac{\gamma_p+\gamma_q}{2}( I_\omega * |v_2|^q)|v_1|^{p-2} v_1 ,\\
			-\Delta v_2 + \mu_2 v_2 &=2_\omega^* ( I_\omega * |v_2|^{2_\omega^*} ) |v_2|^{2_\omega^* -2} v_2 + \theta q\frac{\gamma_p+\gamma_q}{2}( I_\omega * |v_1|^p)|v_2|^{q-2} v_2, \\
			\int_{\mathbb{R}^N} v_1^2 = \alpha_1^2 \, &,  \int_{\mathbb{R}^N} v_2^2 = \alpha_2^2.
		\end{aligned}
		\right.
		\end{equation*}
		
		\par By the Pohozaev identity, we have
		\begin{align*}
			|\nabla v_1|^2_2+|\nabla v_2|^2_2-&2_\omega^* \int_{\mathbb{R}^N}\left((I_\omega*|v_1|^{2_\omega^*})|v_1|^{2_\omega^*}+(I_\omega*|v_2|^{2_\omega^*})|v_2|^{2_\omega^*}\right)\\
			-&\theta\frac{(\gamma_p+\gamma_q)^2}{2}\int_{\mathbb{R}^N}(I_\omega*|v_1|^p)|v_2|^q=0.
		\end{align*}
	
		\par Thus, $(\Psi_{(v_1,v_2)}^\theta)''(1)=0$, $(v_1,v_2) \in \mathcal{P}_\theta^0(\alpha_1,\alpha_2)$, which is in contradiction with\\ $\mathcal{P}_\theta^0(\alpha_1,\alpha_2)=\emptyset$.\\
		$(\romannumeral2)$-$(\romannumeral4)$ Let $(v_1,v_2) \in \mathcal{S}(\alpha_1,\alpha_2)$, $\Psi_{(v_1,v_2)}^\theta(t)=\mathcal{J}_\theta(t\diamond(v_1,v_2)) \geq h \left( (\int_{\mathbb{R}^N}|\nabla v_1|^2+|\nabla v_2|^2)^{\frac{1}{2}} \right)$.\\
		Then for any $t \in (\frac{T_0}{ (\int_{\mathbb{R}^N}|\nabla v_1|^2+|\nabla v_2|^2)^{\frac{1}{2}}}, \frac{T_1}{ (\int_{\mathbb{R}^N}|\nabla v_1|^2+|\nabla v_2|^2)^{\frac{1}{2}}})$, $\Psi_{(v_1,v_2)}^\theta(t)> \varepsilon\alpha_1\alpha_2$.
		
		\par Note that $\lim\limits_{t\to 0^+}\Psi_{(v_1,v_2)}^\theta(t)=-\varepsilon\int_{\mathbb{R}^N}v_1v_2 \leq \varepsilon\alpha_1\alpha_2$,  $\lim\limits_{t\to +\infty}\Psi_{(v_1,v_2)}^\theta(t)=-\infty$, $\Psi_{(v_1,v_2)}^\theta(t)$ is decreasing as $0<t\ll 1$ and $(\Psi_{(v_1,v_2)}^\theta)'(t)=0$ has at most two solutions. Therefore, $\Psi_{(v_1,v_2)}^\theta(t)$ has exactly two critical points $s_\theta(v_1,v_2)<t_\theta(v_1,v_2)$ satisfying
		\[\Psi_{(v_1,v_2)}^\theta(s_\theta(v_1,v_2))=\min\{\Psi_{(v_1,v_2)}^\theta(t):t(\int_{\mathbb{R}^N}|\nabla v_1|^2+|\nabla v_2|^2)^{\frac{1}{2}}\leq T_0\} \]
		 and  $\Psi_{(v_1,v_2)}^\theta(t_\theta(v_1,v_2))=\max\limits_{t>0}\Psi_{(v_1,v_2)}^\theta(t)$.\\
		$(\romannumeral5)$ Denote $\Phi(t,(v_1,v_2))=(\Psi_{(v_1,v_2)}^\theta)'(t)$, then $\Phi(s_\theta(v_1,v_2),(v_1,v_2))=0$ and
		\[\partial_t \Phi(t,(v_1,v_2))|_{t=s_\theta(v_1,v_2)}=(\Psi_{(v_1,v_2)}^\theta)''(s_\theta(v_1,v_2))>0.\]
		Thus, the map $(v_1,v_2) \mapsto s_\theta(v_1,v_2)$ has $C^1$-regularity due to the implicit function theorem.
	\end{proof}
	For $T>0$, we define
	\begin{equation}
		A_T(\alpha_1,\alpha_2):=\{(v_1,v_2) \in \mathcal{S}(\alpha_1,\alpha_2):|\nabla v_1|_2^2+|\nabla v_2|_2^2<T^2\}.
	\end{equation}
	
	The following lemma establishes the properties of $m_\theta(\alpha_1,\alpha_2)$ and its monotonicity.
	\begin{lem}
		Let $0<\theta<\theta_0$, $0<\varepsilon<\varepsilon_*$, then
		\begin{enumerate}[label=(\roman*)]
			\item $m_\theta(\alpha_1,\alpha_2)=\inf\limits_{A_{T_0}(\alpha_1,\alpha_2)}\mathcal{J}_\theta=\inf\limits_{\mathcal{P}_\theta^+(\alpha_1,\alpha_2)}\mathcal{J}_\theta<-\varepsilon\alpha_1\alpha_2$. Moreover, there exists $\delta>0$ satisfying
			\[m_\theta(\alpha_1,\alpha_2)<\inf\limits_{A_{T_0}(\alpha_1,\alpha_2) \setminus A_{T_0-\delta}(\alpha_1,\alpha_2)}\mathcal{J}_\theta.\]
			\item $m_\theta(\alpha_1,\alpha_2) \leq m_\theta(\alpha_1',\alpha_2')$, for every $0<\alpha_1'\leq \alpha_1$, $0<\alpha_2' \leq \alpha_2$.
		\end{enumerate}
		\label{Lemma3.6单调性}
	\end{lem}
	\begin{proof}
		$(\romannumeral1)$ By Lemma \ref{Lemma3.5变分结构}, we obtain $m_\theta(\alpha_1,\alpha_2)=\inf\limits_{\mathcal{P}_\theta^+(\alpha_1,\alpha_2)}\mathcal{J}_\theta$ and $\mathcal{P}_\theta^+(\alpha_1,\alpha_2) \subset A_{T_0}(\alpha_1,\alpha_2)$. Then we have
		 \[m_\theta(\alpha_1,\alpha_2) \geq \inf\limits_{A_{T_0}(\alpha_1,\alpha_2)}\mathcal{J}_\theta.\]
		  
		  \par For $(v_1,v_2) \in A_{T_0}(\alpha_1,\alpha_2)$, since $m_\theta(\alpha_1,\alpha_2)\leq \mathcal{J}_\theta(s_\theta(v_1,v_2)\diamond(v_1,v_2)) \leq \mathcal{J}_\theta(v_1,v_2)$, we get
		  \[m_\theta(\alpha_1,\alpha_2) \leq \inf\limits_{A_{T_0}(\alpha_1,\alpha_2)}\mathcal{J}_\theta.\]
		  
		  \par The next step is to show $m_\theta(\alpha_1,\alpha_2)<-\varepsilon\alpha_1\alpha_2$. Let $v_1>0$ and $|v_1|_2=\alpha_1$, $v_2:=\frac{v_1}{|v_1|_2}\alpha_2$, we can find $0<t\ll 1$ satisfying $t\diamond(v_1,v_2) \in A_{T_0}(\alpha_1,\alpha_2)$. Thus, 
		  {\small
		  \begin{align*}
		  	m_\theta(\alpha_1,\alpha_2)\leq& \mathcal{J}_\theta(t\diamond(v_1,v_2))=\frac{t^2}{2}(|\nabla v_1|^2_2+\frac{\alpha_2^2}{\alpha_1^2}|\nabla v_1|^2_2)-\frac{t^{22_\omega^*}}{22_\omega^*}(1+\frac{\alpha_2^{22_\omega^*}}{\alpha_1^{22_\omega^*}})\int_{\mathbb{R}^N}(I_\omega*|v_1|^{2_\omega^*})|v_1|^{2_\omega^*}\\
		  	&-\theta t^{\gamma_p+\gamma_q}(\frac{\alpha_2}{\alpha_1})^q\int_{\mathbb{R}^N}(I_\omega*|v_1|^p)|v_2|^q-\varepsilon\alpha_1\alpha_2<-\varepsilon\alpha_1\alpha_2, \qquad \text{for } t>0 \text{ small enough}.
		  \end{align*}
		  }
		  
		  \par Finally, notice that $h(T_0)=\varepsilon\alpha_1\alpha_2>0$. We can find $\delta>0$ satisfying $h(s)\geq \frac{1}{2}m_\theta(\alpha_1,\alpha_2)$, for each $s\in [T_0-\delta,T_0]$. Then for $(v_1,v_2) \in A_{T_0}(\alpha_1,\alpha_2) \setminus A_{T_0-\delta}(\alpha_1,\alpha_2)$, we have \[\mathcal{J}_\theta(v_1,v_2)\geq h\left( (|\nabla v_1|_2^2+|\nabla v_2|_2^2)^{\frac{1}{2}} \right)\geq \frac{1}{2}m_\theta(\alpha_1,\alpha_2)>m_\theta(\alpha_1,\alpha_2).\]
		  $(\romannumeral2)$The argument resembles the one used in \cite[Lemma 2.3]{bartsch2023existence}, with slight modifications in our case. 
		  Recall $s_*$ defined in \eqref{s_*}. Similarly to $(\romannumeral1)$, we have
		  $m_\theta(\alpha_1,\alpha_2) = \inf\limits_{A_{s_*}(\alpha_1,\alpha_2)}\mathcal{J}_\theta$.\\
		  Then for every $\epsilon >0$, one can find $(v_1,v_2)\in A_{s_*}(\alpha_1',\alpha_2')$ satisfying
		  \begin{equation}
		  	\mathcal{J}_\theta(v_1,v_2)\leq m_\theta(\alpha_1',\alpha_2')+\frac{\epsilon}{3}.
		  	\label{3.11}
		  \end{equation}
		  
		  \par We choose a radial bump function $\varphi\in C_0^\infty(\mathbb{R}^N)$, taking values in $[0,1]$, such that $\varphi|_{B(0,1)}\equiv 1$, $\supp \varphi \subset B(0,2)$, where
		   $B(0,R):=\{x \in \mathbb{R}^N:|x|\leq R\}$. 
		   
		   \par For $\delta>0$, let $v_{1,\delta}:=v_1(x)\varphi(\delta x)$, $v_{2,\delta}:=v_2(x)\varphi(\delta x)$, thus $(v_{1,\delta},v_{2,\delta})\to (v_1,v_2)$ in $H$ as $\delta\to 0^+$. Given $0<\eta \ll s_*- (|\nabla v_1|_2^2+|\nabla v_2|_2^2)^{\frac{1}{2}}$, one can find $\delta>0$ small satisfying
		   \begin{equation}
		   	\mathcal{J}_\theta(v_{1,\delta},v_{2,\delta}) \leq \mathcal{J}_\theta (v_1,v_2) + \frac{\epsilon}{3} \; \text{and}\; (|\nabla v_{1,\delta}|_2^2+|\nabla v_{2,\delta}|_2^2)^{\frac{1}{2}}<s_*-\eta.
		   	\label{3.12}
		   \end{equation}
		   
		   \par Let $\psi_1,\psi_2 \in C_0^\infty(\mathbb{R}^N)$ and $\supp \psi_1 \subset B(0,1+\frac{4}{\delta}) \setminus B(0,\frac{4}{\delta})$, $\supp \psi_2 \subset B(0,2+\frac{4}{\delta}) \setminus B(0,1+\frac{4}{\delta})$. 
		   Define \[\phi_{\alpha_1}=\frac{\sqrt{\alpha_1^2-|v_{1,\delta}|_2^2}}{|\psi_1|_2}\psi_1 \;\text{and}\; \phi_{\alpha_2}=\frac{\sqrt{\alpha_2^2-|v_{2,\delta}|_2^2}}{|\psi_2|_2}\psi_2.\]
		   Thus, we have
		   \[ \left( \supp(t\diamond \phi_{\alpha_1}) \cup \supp(t\diamond \phi_{\alpha_2}) \right)\cap (\supp v_{1,\delta} \cup \supp v_{2,\delta} )=\emptyset, \qquad \text{for} \;\; 0<t<1.
		   \]
		   \[\supp(t\diamond \phi_{\alpha_1}) \cap \supp(t\diamond \phi_{\alpha_2})=\emptyset.
		   \]
		   
		   \par Hence $(v_{1,\delta}+t\diamond \phi_{\alpha_1},v_{2,\delta}+t\diamond \phi_{\alpha_2})\in \mathcal{S}(\alpha_1,\alpha_2)$. 
		   Notice that
		   $\mathcal{J}_\theta(t\diamond(\phi_{\alpha_1},\phi_{\alpha_2})) \to 0 $, $(|\nabla (t\diamond \phi_{\alpha_1})|_2^2+|\nabla (t\diamond \phi_{\alpha_2})|_2^2)^{\frac{1}{2}} \to 0$ as $t \to 0^+$. 
		   One can find $0<t\ll 1$ satisfying
		   \begin{equation}
		   	\mathcal{J}_\theta(t\diamond(\phi_{\alpha_1},\phi_{\alpha_2})) \leq \frac{\epsilon}{3} \;\; \text{and} \;\;
		   	(|\nabla (v_{1,\delta}+t\diamond \phi_{\alpha_1})|_2^2+|\nabla (v_{2,\delta}+t\diamond \phi_{\alpha_2})|_2^2)^{\frac{1}{2}} <s_*.
		   \end{equation}
		   
		   \par Therefore, by \eqref{3.11} and \eqref{3.12}, we have
		   {\small
		   \begin{align*}
		   	m_\theta(\alpha_1,\alpha_2)& \leq \mathcal{J}_\theta(v_{1,\delta}+t\diamond \phi_{\alpha_1},v_{2,\delta}+t\diamond \phi_{\alpha_2}) \\
		   	&=\mathcal{J}_\theta(v_{1,\delta},v_{2,\delta})+\mathcal{J}_\theta(t\diamond\phi_{\alpha_1},t\diamond\phi_{\alpha_2})-\frac{1}{2_\omega^*}t^{2_\omega^*}\int_{\mathbb{R}^N}(I_\omega*|v_{1,\delta}|^{2_\omega^*})|\phi_{\alpha_1}|^{2_\omega^*}\\
		   	&-\frac{1}{2_\omega^*}t^{2_\omega^*}\int_{\mathbb{R}^N}(I_\omega*|v_{2,\delta}|^{2_\omega^*})|\phi_{\alpha_2}|^{2_\omega^*}-\theta t^p\int_{\mathbb{R}^N}(I_\omega*|\phi_{\alpha_1}|^p)|v_{2,\delta}|^q-\theta t^q\int_{\mathbb{R}^N}(I_\omega*|v_{1,\delta}|^p)|\phi_{\alpha_2}|^q\\
		   	&\leq\mathcal{J}_\theta(v_{1,\delta},v_{2,\delta})+\mathcal{J}_\theta(t\diamond\phi_{\alpha_1},t\diamond\phi_{\alpha_2})\\
		   	&\leq m_\theta(\alpha_1',\alpha_2')+\epsilon.
		   \end{align*}
		   }
		   Since $\epsilon>0$ is arbitrary, it follows that $m_\theta(\alpha_1,\alpha_2) \leq m_\theta(\alpha_1',\alpha_2')$.
	\end{proof}
	
	Next, we establish the compactness result.
	\begin{lem}
		Assume $N\in \{3,4\}$, if  $\{(v_{1,n},v_{2,n})\}\subset \mathcal{S}(\alpha_1,\alpha_2)$ satisfies that as $n\to \infty$
		\begin{align}
			\mathcal{J}_\theta'(v_{1,n},v_{2,n})+\mu_{1,n}&v_{1,n}+\mu_{2,n}v_{2,n}\to 0,  \qquad \text{for some}\; \mu_{1,n},\mu_{2,n}\in \mathbb{R}, \label{3.6}\\
			\mathcal{J}_\theta(v_{1,n},v_{2,n})&\to m_\theta(\alpha_1,\alpha_2), \qquad 	\mathcal{P}_\theta(v_{1,n},v_{2,n})\to 0, \label{3.7}\\
			&v_{1,n}^-, v_{2,n}^- \to 0 ~~ \text{a.e in}\; \mathbb{R}^N, \label{3.8}
		\end{align}
		with
		\begin{equation}
			c:=m_\theta(\alpha_1,\alpha_2)<-\varepsilon\alpha_1\alpha_2. \label{3.9}
		\end{equation}
		Then there exists $(v_1,v_2)\in H$, $v_1,v_2>0$ and $\mu_1,\mu_2>0$ satisfying that, for a subsequence,
		$(v_{1,n},v_{2,n})\to (v_1,v_2)\; \text{in}\; H$ and $(\mu_{1,n},\mu_{2,n}) \to (\mu_1,\mu_2)\; \text{in} \; \mathbb{R}^2$.
		\label{Lemma3.7紧性}
	\end{lem}
	\begin{proof}
		The argument is structured in three stages.

			\textbf{Step 1:} The sequences $\{(v_{1,n},v_{2,n})\}$ and $(\mu_{1,n},\mu_{2,n})$ are shown to be bounded in $H$ and $\mathbb{R}^2$, respectively.
			Since $22_\omega^*<p+q<\frac{2N+2\omega+4}{N}$, one has $0<\gamma_p+\gamma_q<2$. By \eqref{3.7}, we obtain
			{\small
			\begin{align*}
				c+o_n(1)=&\mathcal{J}_\theta(v_{1,n},v_{2,n})-\frac{1}{22_\omega^*}P_\theta(v_{1,n},v_{2,n}) \geq \frac{2_\omega^*-1}{22_\omega^*}(|\nabla v_{1,n}|_2^2+|\nabla v_{2,n}|_2^2)\\
				&-\frac{22_\omega^*-\gamma_p-\gamma_q}{22_\omega^*}\theta C(N,p,q)(\alpha_1^2+\alpha_2^2)^{\frac{p+q-\gamma_p-\gamma_q}{2}}(|\nabla v_{1,n}|_2^2+|\nabla v_{2,n}|_2^2)^{\frac{\gamma_p+\gamma_q}{2}}-\varepsilon\alpha_1\alpha_2.
			\end{align*}
		}
			
			\par Hence, $\{(v_{1,n},v_{2,n})\}$ is bounded in $H$. Based on \eqref{3.6}, it can be inferred that
			\[ \mu_{1,n}=-\frac{1}{\alpha_1^2}\mathcal{J}_\theta'(v_{1,n},v_{2,n})[(v_{1,n},0)]+o_n(1), \qquad \mu_{2,n}=-\frac{1}{\alpha_2^2}\mathcal{J}_\theta'(v_{1,n},v_{2,n})[(0,v_{2,n})]+o_n(1).\]
			Thus, $(\mu_{1,n},\mu_{2,n})$ is bounded in $\mathbb{R}^2$. Then there exist $(v_1,v_2) \in H$, $(\mu_1,\mu_2) \in \mathbb{R}^2$, up to a subsequence, such that
			\begin{align*}
				&(v_{1,n},v_{2,n})\rightharpoonup (v_1,v_2)\quad in  \; H,\\
				&(v_{1,n},v_{2,n})\to (v_1,v_2) \quad in  \; L^r(\mathbb{R}^N)\times L^r(\mathbb{R}^N),~r\in (2,2^*)\\
				&(v_{1,n},v_{2,n})\to (v_1,v_2) \quad a.e \; in  \; \mathbb{R}^N,\\
				&(\mu_{1,n},\mu_{2,n}) \to (\mu_1,\mu_2) \quad in  \; \mathbb{R}^N.
			\end{align*}
			
			\par In addition, it follows from \eqref{3.6} and \eqref{3.8} that
			\begin{align}
				\mathcal{J}_\theta'(v_1,v_2)+\mu_1 v_1+\mu_2 v_2&=0, \qquad v_1\geq 0,\; v_2\geq 0. \label{3.10}\\
				P_\theta(v_1,v_2)&=0 .   
				\label{Pbeta(u1,u2)0)}
			\end{align}
			
			\textbf{Step 2:} We intend to show $v_1 \neq 0$, $v_2 \neq 0$. Then $v_1,v_2>0$ by the maximum principle.
			To derive a contradiction, we may suppose $v_1=0$. By \eqref{3.10}, we conclude that
			\begin{equation*}
				\left\{
				\begin{aligned}
					0&=\varepsilon v_2,\\
					-\Delta v_2+\mu_2v_2&=( I_\omega * |v_2|^{2_\omega^*} ) |v_2|^{2_\omega^* -2} v_2 .
				\end{aligned}
				\right.
			\end{equation*}
			Thus, we get $v_2=0$.
			With no loss of generality, we assume $|\nabla v_{1,n}|_2^2 \to r_1 \geq 0$, $|\nabla v_{2,n}|_2^2 \to r_2 \geq 0$. 
			If $r_1=r_2=0$, by \eqref{GN} and \eqref{SH}, we have
			{\small
			\begin{align*}
				c+o_n(1)=&\mathcal{J}_\theta(v_{1,n},v_{2,n})\\
				=&\frac{1}{2}(|\nabla v_{1,n}|^2_2+|\nabla v_{2,n}|^2_2)-\frac{1}{22_\omega^*}\int_{\mathbb{R}^N}\left((I_\omega*|v_{1,n}|^{2_\omega^*})|v_{1,n}|^{2_\omega^*}+(I_\omega*|v_{2,n}|^{2_\omega^*})|v_{2,n}|^{2_\omega^*}\right)\\
				&-\theta \int_{\mathbb{R}^N}(I_\omega*|v_{1,n}|^p)|v_{2,n}|^q-\varepsilon \int_{\mathbb{R}^N}v_{1,n}v_{2,n} \geq -\varepsilon\alpha_1\alpha_2 +o_n(1),
			\end{align*}
		}
			which contradicts $c<-\varepsilon\alpha_1\alpha_2 $. Therefore, $r_1+r_2>0$. 
			
			\par Due to the weak convergence $(v_{1,n},v_{2,n})\rightharpoonup (v_1,v_2)$ in $H$, we conclude from  \cite[Lemma 2.5]{zhang2023normalizedsolutionscriticalchoquard} that
			\begin{equation}
			\int_{\mathbb{R}^N}(I_\omega*|v_{1,n}|^p)|v_{2,n}|^q=\int_{\mathbb{R}^N}(I_\omega*|v_1|^p)|v_2|^q+o_n(1).
			\label{2.12}
			\end{equation}
			Using \eqref{2.12}, we get $\int_{\mathbb{R}^N}(I_\omega*|v_{1,n}|^p)|v_{2,n}|^q \to 0$. By \eqref{SH} and \eqref{3.7}, we have
			\[r_1+r_2\leq S_H^{-2_\omega^*}(r_1^{2_\omega^*}+r_2^{2_\omega^*}) \leq  S_H^{-2_\omega^*}(r_1+r_2)^{2_\omega^*}.\]
			Then $r_1+r_2 \geq S_H^{\frac{2_\omega^*}{2_\omega^*-1}}$. Thus, \[c=\lim\limits_{n \to +\infty}\mathcal{J}_\theta(v_{1,n},v_{2,n}) \geq \frac{2_\omega^*-1}{22_\omega^*}S_H^{\frac{2_\omega^*}{2_\omega^*-1}}-\varepsilon\alpha_1\alpha_2,\] which contradicts \eqref{3.9}.
			
			\textbf{Step 3:}  We prove that $(v_{1,n},v_{2,n})\to (v_1,v_2)$ in $H$.
			If $\mu_1 \leq 0$, then
			\[-\Delta v_1 =- \mu_1 v_1+ ( I_\omega * |v_1|^{2_\omega^*} ) |v_1|^{2_\omega^* -2} v_1 + \theta p( I_\omega * |v_2|^q)|v_1|^{p-2} v_1 + \varepsilon v_2 \geq 0, \qquad in \,\, \mathbb{R}^N.\]
			By Lemma \ref{smooth} and \ref{Liouville}, we get $v_1 \equiv 0$ in $\mathbb{R}^N$, which is in contradiction with step 2. Therefore $\mu_1 >0$ and similarly $\mu_2>0$.
			Let $(\overline{v}_{1,n},\overline{v}_{2,n}):=(v_{1,n}-v_1,v_{2,n}-v_2)$. Recalling the Brezis-Lieb type lemma \cite[Lemma 3.2]{Yang2013} and \cite[Lemma 2.2]{chen2021normalized}, we obtain the following equalities:
			{\small
			\begin{align}
				\int_{\mathbb{R}^N}(I_\omega*|v_{1,n}|^{2_\omega^*})|v_{1,n}|^{2_\omega^*}&=\int_{\mathbb{R}^N}(I_\omega*|v_1|^{2_\omega^*})|v_1|^{2_\omega^*}+\int_{\mathbb{R}^N}(I_\omega*|\overline{v}_{1,n}|^{2_\omega^*})|\overline{v}_{1,n}|^{2_\omega^*}+o_n(1)\label{Brezis-Lieb of Riesz potential}, \\
				\int_{\mathbb{R}^N}(I_\omega*|v_{2,n}|^{2_\omega^*})|v_{2,n}|^{2_\omega^*}&=\int_{\mathbb{R}^N}(I_\omega*|v_2|^{2_\omega^*})|v_2|^{2_\omega^*}+\int_{\mathbb{R}^N}(I_\omega*|\overline{v}_{2,n}|^{2_\omega^*})|\overline{v}_{2,n}|^{2_\omega^*}+o_n(1),\\
				\int_{\mathbb{R}^N}v_{1,n}v_{2,n}&=\int_{\mathbb{R}^N}v_1v_2+\int_{\mathbb{R}^N}\overline{v}_{1,n}\overline{v}_{2,n}+o_n(1). \label{Brezis-Lieb of linear couplings}
			\end{align}
		}
			
			\par By \eqref{3.7}, \eqref{Pbeta(u1,u2)0)} and \eqref{2.12}, we obtain
			{\small
			\begin{equation}
				|\nabla \overline{v}_{1,n}|^2_2+|\nabla\overline{v}_{2,n}|^2_2 =\int_{\mathbb{R}^N}(I_\omega*|\overline{v}_{1,n}|^{2_\omega^*})|\overline{v}_{1,n}|^{2_\omega^*}+\int_{\mathbb{R}^N}(I_\omega*|\overline{v}_{2,n}|^{2_\omega^*})|\overline{v}_{2,n}|^{2_\omega^*}+o_n(1).
				\label{2.20}
			\end{equation}
		}
			
			\par We claim that $|\nabla \overline{v}_{1,n}|^2_2+|\nabla\overline{v}_{2,n}|^2_2 \to 0$. Otherwise, similar to step 2, we have
			\[|\nabla \overline{v}_{1,n}|^2_2+|\nabla\overline{v}_{2,n}|^2_2 \to r \geq  S_H^{\frac{2_\omega^*}{2_\omega^*-1}} .	\]
			Note that $v_1,v_2>0$, $0<|v_1|_2 \leq \alpha_1$, $0<|v_2|_2 \leq \alpha_2$. By Lemma \ref{Lemma3.6单调性} and \eqref{2.20}, we get
			\begin{align*}
				m_\theta(\alpha_1,\alpha_2)+o_n(1)=&\mathcal{J}_\theta(v_{1,n},v_{2,n})\\
				=&\mathcal{J}_\theta(v_1,v_2)+\frac{1}{2}(|\nabla \overline{v}_{1,n}|^2_2+|\nabla\overline{v}_{2,n}|^2_2)-\frac{1}{22_\omega^*}\int_{\mathbb{R}^N} (I_\omega*|\overline{v}_{1,n}|^{2_\omega^*})|\overline{v}_{1,n}|^{2_\omega^*}\\
				&-\frac{1}{22_\omega^*}\int_{\mathbb{R}^N}(I_\omega*|\overline{v}_{2,n}|^{2_\omega^*})|\overline{v}_{2,n}|^{2_\omega^*} -\varepsilon\int_{\mathbb{R}^N}v_{1,n}v_{2,n}+\varepsilon\int_{\mathbb{R}^N}v_1v_2\\
				\geq & m_\theta(|v_1|_2,|v_2|_2)+\frac{2_\omega^*-1}{22_\omega^*}(|\nabla \overline{v}_{1,n}|^2_2+|\nabla\overline{v}_{2,n}|^2_2)-\varepsilon\alpha_1\alpha_2+o_n(1)\\
				\geq & m_\theta(\alpha_1,\alpha_2)+\frac{2_\omega^*-1}{22_\omega^*} S_H^{\frac{2_\omega^*}{2_\omega^*-1}} -\varepsilon\alpha_1\alpha_2+o_n(1),
			\end{align*}
			which contradicts \eqref{3.9}. Therefore, $|\nabla \overline{v}_{1,n}|^2_2+|\nabla\overline{v}_{2,n}|^2_2 \to 0$.
			
			\par Let $|v_1|_2=\overline{\alpha}_1 \in (0,\alpha_1]$, $|v_2|_2=\overline{\alpha}_2 \in (0,\alpha_2]$. Applying the test function $(v_{1,n},v_{2,n})$ and $(v_1,v_2)$ to \eqref{3.6} and \eqref{3.10} respectively, by \eqref{2.12} and \eqref{Brezis-Lieb of Riesz potential}, we obtain
			\[\mu_1(\alpha_1^2-\overline{\alpha}_1^2)+\mu_2(\alpha_2^2-\overline{\alpha}_2^2)=2\varepsilon \lim\limits_{n \to +\infty} \int_{\mathbb{R}^N}(v_{1,n}-v_1)(v_{2,n}-v_2).\]
			
			\par We claim that $\overline{\alpha}_1=\alpha_1$, $\overline{\alpha}_2=\alpha_2$. Otherwise, we may assume $0<\overline{\alpha}_1<\alpha_1$.\\
			If $\overline{\alpha}_2=\alpha_2$, we have
			\[\mu_1(\alpha_1^2-\overline{\alpha}_1^2)=2\varepsilon \lim\limits_{n \to +\infty} \int_{\mathbb{R}^N}(v_{1,n}-v_1)(v_{2,n}-v_2) \leq 2\varepsilon \sqrt{(\alpha_1^2-\overline{\alpha}_1^2)(\alpha_2^2-\overline{\alpha}_2^2)}=0,
			\]
			which contradicts $\overline{\alpha}_1<\alpha_1$.\\
			If $0<\overline{\alpha}_2<\alpha_2$, we get
			\begin{align*}
				2\sqrt{\mu_1\mu_2}\sqrt{(\alpha_1^2-\overline{\alpha}_1^2)(\alpha_2^2-\overline{\alpha}_2^2)}&\leq \mu_1(\alpha_1^2-\overline{\alpha}_1^2)+\mu_2(\alpha_2^2-\overline{\alpha}_2^2)\\
				&=2\varepsilon \lim\limits_{n \to +\infty} \int_{\mathbb{R}^N}(v_{1,n}-v_1)(v_{2,n}-v_2)\\
				&\leq 2\varepsilon \sqrt{(\alpha_1^2-\overline{\alpha}_1^2)(\alpha_2^2-\overline{\alpha}_2^2)}.
			\end{align*}
			Hence $\varepsilon\geq \sqrt{\mu_1\mu_2}$. Let $(\overline{v}_1,\overline{v}_2):=(\sqrt{\mu_2}v_1,\sqrt{\mu_1}v_2)$, we have 
			\begin{equation*}
				\left\{
				\begin{aligned}
					-\Delta \overline{v}_1+\mu_1 \overline{v}_1 \geq \mu_2 \overline{v}_2,\\
					-\Delta \overline{v}_2+\mu_2 \overline{v}_2 \geq \mu_1 \overline{v}_1.
				\end{aligned}
				\right.
			\end{equation*}
			Then $-\Delta(\overline{v}_1+\overline{v}_2) \geq 0$. By  Lemma \ref{Liouville}, we obtain $\overline{v}_1 \equiv 0$, $\overline{v}_2 \equiv 0 $, which is in contradiction with step 2.
			Thus, $(v_{1,n},v_{2,n})\to (v_1,v_2)$ in $H$.
	\end{proof}
	
	\begin{proof}[Proof of Theorem 1.1]\leavevmode\par
		 Let $(\overline{v}_{1,n},\overline{v}_{2,n}) \in  A_{T_0}(\alpha_1,\alpha_2)$ be a minimizing sequence of $m_\theta(\alpha_1,\alpha_2)$.
		Notice that \[\mathcal{J}_\theta (|\overline{v}_{1,n}|,|\overline{v}_{2,n}|) \leq  \mathcal{J}_\theta(\overline{v}_{1,n},\overline{v}_{2,n}).\] 
		We can suppose that $\overline{v}_{1,n},\overline{v}_{2,n} \geq 0$. Using Lemma \ref{Lemma3.5变分结构}, let $(\widetilde{v}_{1,n},\widetilde{v}_{2,n}):=s_\theta (\overline{v}_{1,n},\overline{v}_{2,n}) \diamond (\overline{v}_{1,n},\overline{v}_{2,n})$. Then $\mathcal{J}_\theta(\widetilde{v}_{1,n},\widetilde{v}_{2,n}) \leq  \mathcal{J}_\theta (\overline{v}_{1,n},\overline{v}_{2,n})$ and $(\widetilde{v}_{1,n},\widetilde{v}_{2,n}) \in \mathcal{P}_\theta^+(\alpha_1,\alpha_2)$.\\
		By Ekeland's variational principle, there exist Palais-Smale sequence $(v_{1,n},v_{2,n})$ for\\ $\mathcal{J}_\theta|_{\mathcal{S}(\alpha_1,\alpha_2)}$ such that \[||(v_{1,n},v_{2,n})-(\widetilde{v}_{1,n},\widetilde{v}_{2,n})||_H \to 0, \qquad \text{as}\; n \to +\infty. \]
		
		\par Thus $P_\theta (v_{1,n},v_{2,n}) =P_\theta(\widetilde{v}_{1,n},\widetilde{v}_{2,n}) + o_n(1)=o_n(1)$ and $v_{1,n}^-,v_{2,n}^- \to 0$ a.e in $\mathbb{R}^N$. By Lemma \ref{Lemma3.7紧性}, we have
		$(v_{1,n},v_{2,n})\to (v_1,v_2)$ in $H$ and $(\mu_{1,n},\mu_{2,n}) \to (\mu_1,\mu_2)$ in $\mathbb{R}^2$ with $\mu_1,\mu_2 >0$. Therefore, we obtain a normalized ground state $(v_1,v_2)$ of system \eqref{system1.1}.
	\end{proof}

	\subsection{The $L^2$-supercritical case $p+q>\frac{2N+2\omega+4}{N}$}

	\begin{lem}
		For $\theta >0$ and for any $(v_1,v_2) \in \mathcal{S}(\alpha_1,\alpha_2)$, $\Psi_{(v_1,v_2)}^\theta(t)$ admits precisely one critical point $t_\theta(v_1,v_2)$ satisfying
		\[\Psi_{(v_1,v_2)}^\theta(t_\theta(v_1,v_2))=\max_{t>0} \Psi_{(v_1,v_2)}^\theta(t) > -\varepsilon\alpha_1\alpha_2 . \]
		Moreover,
		\begin{enumerate}[label=(\roman*)]
			\item $\mathcal{P}_\theta(\alpha_1,\alpha_2)=\mathcal{P}_\theta^-(\alpha_1,\alpha_2)$,  which is a submanifold of $H$.
			\item $t \diamond (v_1,v_2) \in \mathcal{P}_\theta(\alpha_1,\alpha_2)$ exactly when 
			$t= t_\theta(v_1,v_2)$.
			\item $\Psi_{(v_1,v_2)}^\theta(t)$ decreases strictly and exhibits concavity on $(t_\theta(v_1,v_2),+\infty)$.
			\item The function $(v_1,v_2) \mapsto t_\theta(v_1,v_2)$ is of class $C^1$.
		\end{enumerate}
		\label{Lemma 3.8 变分结构}
	\end{lem}
	\begin{proof}
		$(\romannumeral1)$ To derive a contradiction, assume that $(v_1,v_2) \in \mathcal{P}_\theta(\alpha_1,\alpha_2)\setminus \mathcal{P}_\theta^-(\alpha_1,\alpha_2)$. Then
		\begin{align*}
			(\Psi_{(v_1,v_2)}^\theta)'(1)=&|\nabla v_1|^2_2+|\nabla v_2|^2_2- \int_{\mathbb{R}^N}\left((I_\omega*|v_1|^{2_\omega^*})|v_1|^{2_\omega^*}+(I_\omega*|v_2|^{2_\omega^*})|v_2|^{2_\omega^*}\right)\\
			&-\theta(\gamma_p+\gamma_q)\int_{\mathbb{R}^N}(I_\omega*|v_1|^p)|v_2|^q=0,
		\end{align*}
		\begin{align*}
			(\Psi_{(v_1,v_2)}^\theta)''(1)=&2(|\nabla v_1|^2_2+|\nabla v_2|^2_2)- 2_\omega^*\int_{\mathbb{R}^N}\left((I_\omega*|v_1|^{2_\omega^*})|v_1|^{2_\omega^*}+(I_\omega*|v_2|^{2_\omega^*})|v_2|^{2_\omega^*}\right)\\
			&-\theta(\gamma_p+\gamma_q)^2 \int_{\mathbb{R}^N}(I_\omega*|v_1|^p)|v_2|^q \geq 0.
		\end{align*}
		Thus,
		{\small
		\begin{equation*}
			(22_\omega^*-2)\int_{\mathbb{R}^N}\left((I_\omega*|v_1|^{2_\omega^*})|v_1|^{2_\omega^*}+(I_\omega*|v_2|^{2_\omega^*})|v_2|^{2_\omega^*}\right) \leq -\theta (\gamma_p+\gamma_q)(\gamma_p+\gamma_q-2) \int_{\mathbb{R}^N}(I_\omega*|v_1|^p)|v_2|^q,
		\end{equation*}
	}
		which contradicts $\gamma_p+\gamma_q>2$. Analogous to Lemma \ref{Lemma3.5变分结构},  $\mathcal{P}_\theta(\alpha_1,\alpha_2)$ lies in $H$ as a submanifold.\\
		$(\romannumeral2)$-$(\romannumeral4)$ Notice that $\lim\limits_{t \to 0^+}\Psi_{(v_1,v_2)}^\theta(t)=-\varepsilon\int_{\mathbb{R}^N}v_1v_2$, $\lim\limits_{t \to +\infty}\Psi_{(v_1,v_2)}^\theta(t)=-\infty$ and $\Psi_{(v_1,v_2)}^\theta(t)$ is increasing as $0<t\ll1$.
		Therefore, $\Psi_{(v_1,v_2)}^\theta(t)$ has a global maximum point $t_\theta(v_1,v_2)$. 
		
		\par Moreover, $\frac{1}{t}(\Psi_{(v_1,v_2)}^\theta)'(t)$ is decreasing in $(0, +\infty)$, $\Psi_{(v_1,v_2)}^\theta(t)$ has exactly one critical point. In addition, $t_\theta(v_1,v_2)\diamond (v_1,v_2) \in \mathcal{P}_\theta^-(\alpha_1,\alpha_2)$, $(\Psi_{(v_1,v_2)}^\theta)''(t_\theta(v_1,v_2))<0$. Since $(\Psi_{(v_1,v_2)}^\theta)''(t)$ is decreasing in $(0,+\infty)$, $(\Psi_{(v_1,v_2)}^\theta)''(t)<0$ for every $t>t_\theta(v_1,v_2)$. Finally, similar to Lemma \ref{Lemma3.5变分结构}, the map $(v_1,v_2) \mapsto t_\theta(v_1,v_2)$ has $C^1$-regularity due to the implicit function theorem.
	\end{proof}
	
	We now present the properties of $m_\theta(\alpha_1,\alpha_2)$ and its monotonicity.
	\begin{lem}
		\begin{enumerate}[label=(\roman*)]
			\item One can find $\xi>0$ satisfying
			\[m_\theta(\alpha_1,\alpha_2)=\inf\limits_{\mathcal{S}(\alpha_1,\alpha_2)}\max_{t>0}\mathcal{J}_\theta(t\diamond(v_1,v_2))=\inf\limits_{\gamma \in \Gamma_\theta}\max_{t\in [0,1]}\mathcal{J}_\theta(\gamma(t))>-\varepsilon \alpha_1\alpha_2,	\] 
			\[
			\Gamma_\theta :=\{\gamma:[0,1]\to \mathcal{S}(\alpha_1,\alpha_2) : \gamma~ \text{is continuous}, ~\gamma (0)\in A_\xi(\alpha_1,\alpha_2),~ \mathcal{J}_\theta (\gamma(1)) \leq -\varepsilon \alpha_1 \alpha_2\}.\]
			\item $m_\theta(\alpha_1,\alpha_2) \leq m_\theta(\alpha_1',\alpha_2')$, for every $0<\alpha_1'\leq \alpha_1$, $0<\alpha_2' \leq \alpha_2$.
			\item $m_\theta(\alpha_1,\alpha_2)$ is decreasing in $\theta \in (0,+\infty)$.
		\end{enumerate}
		\label{Lemma3.9能量值}
	\end{lem}
	\begin{proof}
		By Lemma \ref{Lemma 3.8 变分结构}, we get $m_\theta(\alpha_1,\alpha_2)=\inf\limits_{\mathcal{P}_\theta^-(\alpha_1,\alpha_2)}\mathcal{J}_\theta=\inf\limits_{\mathcal{S}(\alpha_1,\alpha_2)}\max\limits_{t>0}\mathcal{J}_\theta(t\diamond(v_1,v_2))$.
		\par  Denote \[l:=\inf\limits_{\gamma \in \Gamma_\theta}\max\limits_{t\in [0,1]}\mathcal{J}_\theta(\gamma(t)).\]
		Using Lemma \ref{Lemma 3.8 变分结构} again, for $\xi >0$ small enough, we have $P_\theta(v_1,v_2)>0$ for $(v_1,v_2) \in A_\xi(\alpha_1,\alpha_2)$ and 
		$P_\theta(v_1,v_2)<0$ for $(v_1,v_2) \in \mathcal{J}_\theta^{-\varepsilon\alpha_1\alpha_2}:=\{(v_1,v_2)\in \mathcal{S}(\alpha_1,\alpha_2) :\mathcal{J}_\theta(v_1,v_2) \leq -\varepsilon\alpha_1\alpha_2 \}$. Then for each $\gamma \in \Gamma_\theta$, we can find $t_\gamma \in [0,1]$ satisfying $P_\theta (\gamma (t_\gamma))=0$. Then
		\[m_\theta(\alpha_1,\alpha_2)=\inf\limits_{\mathcal{P}_\theta(\alpha_1,\alpha_2)}\mathcal{J}_\theta \leq \mathcal{J}_\theta(\gamma (t_\gamma)) \leq \max\limits_{t \in [0,1]} \mathcal{J}_\theta(\gamma (t)).
		\]
		Thus,  $m_\theta(\alpha_1,\alpha_2) \leq l$.
		
		\par On the other hand, given $(v_1,v_2) \in \mathcal{P}_\theta(\alpha_1,\alpha_2)$, define $\gamma(t):=(tK+\frac{1}{K}) \diamond (v_1,v_2)$  for $K \gg 1$. By Lemma \ref{Lemma 3.8 变分结构}, we have $\gamma \in \Gamma_\theta$ and \[\mathcal{J}_\theta(v_1,v_2)=\mathcal{J}_\theta(\gamma(\frac{1}{K}-\frac{1}{K^2}))=\max\limits_{t \in [0,1]} \mathcal{J}_\theta(\gamma (t)) \geq l.\] Thus, $m_\theta(\alpha_1,\alpha_2) \geq l$.
		
		\par Next we show that $m_\theta(\alpha_1,\alpha_2) > -\varepsilon \alpha_1\alpha_2$. Since $\gamma_p+\gamma_q > 2$, for $(v_1,v_2) \in \mathcal{P}_\theta(\alpha_1,\alpha_2)$, by \eqref{GN} and \eqref{SH}, we obtain
		\begin{align*}
			|\nabla v_1|^2_2+|\nabla v_2|^2_2 \leq& \theta(\gamma_p+\gamma_q) C(N,p,q)(\alpha_1^2+\alpha_2^2)^{\frac{p+q-\gamma_p-\gamma_q}{2}} (|\nabla v_1|_2^2+|\nabla v_2|_2^2)^{\frac{\gamma_p+\gamma_q}{2}}\\
			&+S_H^{-2_\omega^*}(|\nabla v_1|^2_2+|\nabla v_2|^2_2)^{2_\omega^*}.
		\end{align*}
		Hence, $\inf\limits_{\mathcal{P}_\theta(\alpha_1,\alpha_2)}(|\nabla v_1|_2^2+|\nabla v_2|_2^2)>0$. Then
		{\small
		\begin{align*}
			m_\theta(\alpha_1,\alpha_2)=&\inf\limits_{(v_1,v_2) \in \mathcal{P}_\theta(\alpha_1,\alpha_2)}\mathcal{J}_\theta(v_1,v_2)-\frac{1}{\gamma_p+\gamma_q}P_\theta(v_1,v_2)\\
			=&\inf\limits_{(v_1,v_2) \in \mathcal{P}_\theta(\alpha_1,\alpha_2)}\frac{\gamma_p+\gamma_q-2}{2(\gamma_p+\gamma_q)}(|\nabla v_1|_2^2+|\nabla v_2|_2^2)\\
			&+\frac{22_\omega^*-(\gamma_p+\gamma_q)}{22_\omega^*(\gamma_p+\gamma_q)}\int_{\mathbb{R}^N}\left((I_\omega*|v_1|^{2_\omega^*})|v_1|^{2_\omega^*}+(I_\omega*|v_2|^{2_\omega^*})|v_2|^{2_\omega^*}\right)-\varepsilon\int_{\mathbb{R}^N}v_1v_2\\
			>&-\varepsilon\alpha_1\alpha_2.
		\end{align*}
	}
		\par Finally, assume $0<\theta' \leq \theta$, we have
		\[m_\theta(\alpha_1,\alpha_2)=\inf\limits_{\mathcal{S}(\alpha_1,\alpha_2)}\max_{t>0}\mathcal{J}_\theta(t\diamond(v_1,v_2)) \leq \inf\limits_{\mathcal{S}(\alpha_1,\alpha_2)}\max_{t>0}\mathcal{J}_{\theta'}(t\diamond(v_1,v_2)) = m_{\theta'}(\alpha_1,\alpha_2).\]
	\end{proof}

	Define $\mathcal{I}_\theta : \mathcal{S}(\alpha_1,\alpha_2) \to \mathbb{R} \cup \{+\infty\}$ by 
	\[\mathcal{I}_\theta (v_1,v_2)=\max_{t>0} \Psi_{(v_1,v_2)}^\theta (t).\]
	
	To obtain the Palais-Smale sequence, the following lemmas are required.
	\begin{lem}\cite[Lemma 3.1]{chen2021normalizedsolutionsnonlinearschrodinger}
		For any $\epsilon >0$, if $(v_{1,0},v_{2,0}) \in \mathcal{S}(\alpha_1,\alpha_2)$ satisfies $\mathcal{I}_\theta(v_{1,0},v_{2,0}) \leq \inf\limits_{\mathcal{S}(\alpha_1,\alpha_2)}\mathcal{I}_\theta(v_1,v_2)+\epsilon$, then for any $\eta >0$,
		we can find $(v_{1,\eta},v_{2,\eta}) \in \mathcal{S}(\alpha_1,\alpha_2) $, $v_{1,\eta} \geq 0$, $v_{2,\eta} \geq  0$ such that
		\begin{align*}
			&\mathcal{I}_\theta (v_{1,\eta},v_{2,\eta}) \leq \mathcal{I}_\theta(v_{1,0},v_{2,0}),\\
			&||(v_{1,\eta},v_{2,\eta})-(v_{1,0},v_{2,0})||_H \leq \eta,\\
			\mathcal{I}_\theta(v_1,v_2) > \mathcal{I}_\theta (v_{1,\eta},v_{2,\eta}) - \frac{\epsilon}{\eta}&||(v_{1,\eta},v_{2,\eta})-(v_1,v_2)||_H , \quad  \forall  (v_1,v_2)\in \mathcal{S}(\alpha_1,\alpha_2) \setminus \{(v_{1,\eta},v_{2,\eta})\}.
		\end{align*}
	\end{lem}

	\begin{lem}
		There exists a radially symmetric (PS) sequence $\{(v_{1,n},v_{2,n})\} \subset \mathcal{S}(\alpha_1,\alpha_2)$, $v_{1,n} \geq 0$, $v_{2,n} \geq 0$ such that as $n \to +\infty$
		\begin{align*}
			\mathcal{J}_\theta(v_{1,n},v_{2,n}) \to m_\theta(\alpha_1,\alpha_2),\\
			\mathcal{J}_\theta '|_{\mathcal{S}(\alpha_1,\alpha_2)} (v_{1,n},v_{2,n}) \to 0.
		\end{align*}
		Moreover, we can improve the condition that $\{(v_{1,n},v_{2,n})\} \subset \mathcal{P}_\theta(\alpha_1,\alpha_2)$. 
		\label{Lemma3.11 PS sequence}
	\end{lem}
	
	Given the Palais-Smale sequence, we are now ready to demonstrate its compactness.
	\begin{lem}
		Assume $N\in \{3,4\}$, if  $\{(v_{1,n},v_{2,n})\}\subset \mathcal{P}_\theta(\alpha_1,\alpha_2)$ satisfies that as $n\to \infty$
		\begin{align}
			\mathcal{J}_\theta'(v_{1,n},v_{2,n})+\mu_{1,n}&v_{1,n}+\mu_{2,n}v_{2,n}\to 0,  \qquad \text{for some}\; \mu_{1,n},\mu_{2,n}\in \mathbb{R}, \label{3.16}\\
			\mathcal{J}_\theta(v_{1,n},v_{2,n})&\to m_\theta(\alpha_1,\alpha_2), \qquad 	\mathcal{P}_\theta(v_{1,n},v_{2,n})= 0, \label{3.17}\\
			&v_{1,n}, v_{2,n} \geq 0,  \label{3.18}
		\end{align}
		with
		\begin{equation}
			c:=m_\theta(\alpha_1,\alpha_2)<\frac{2_\omega^*-1}{22_\omega^*}S_H^{\frac{2_\omega^*}{2_\omega^*-1}}-\varepsilon\alpha_1\alpha_2. \label{3.19}
		\end{equation}
		Then there exists $(v_1,v_2)\in H$, $v_1,v_2>0$ and $\mu_1,\mu_2>0$ satisfying that, for a subsequence, 
		$(v_{1,n},v_{2,n})\to (v_1,v_2)\; in\; H$ and $(\mu_{1,n},\mu_{2,n}) \to (\mu_1,\mu_2)\; in \; \mathbb{R}^2$.
		\label{Lemma3.12紧性}
	\end{lem}
	\begin{proof}
		The argument is structured in three stages.

			\textbf{Step 1:} The sequences $\{(v_{1,n},v_{2,n})\}$ and $(\mu_{1,n},\mu_{2,n})$ are shown to be bounded in $H$ and $\mathbb{R}^2$, respectively.\\
			Since $p+q>\frac{2N+2\omega+4}{N}$, one has $\gamma_p+\gamma_q>2$. By \eqref{3.17}, we obtain
			\begin{align*}
				c+o_n(1)=&\mathcal{J}_\theta(v_{1,n},v_{2,n})-\frac{1}{2}P_\theta(v_{1,n},v_{2,n})\\
				\geq &(\frac{1}{2}-\frac{1}{22_\omega^*}) \int_{\mathbb{R}^N}\left((I_\omega*|v_{1,n}|^{2_\omega^*})|v_{1,n}|^{2_\omega^*}+(I_\omega*|v_{2,n}|^{2_\omega^*})|v_{2,n}|^{2_\omega^*}\right) \\
				&+\theta\frac{\gamma_p+\gamma_q-2}{2}\int_{\mathbb{R}^N}(I_\omega*|v_{1,n}|^p)|v_{2,n}|^q-\varepsilon\alpha_1\alpha_2.
			\end{align*}
			Hence, $\int_{\mathbb{R}^N}(I_\omega*|v_{1,n}|^{2_\omega^*})|v_{1,n}|^{2_\omega^*}$, $\int_{\mathbb{R}^N}(I_\omega*|v_{2,n}|^{2_\omega^*})|v_{2,n}|^{2_\omega^*}$ and $\int_{\mathbb{R}^N}(I_\omega*|v_{1,n}|^p)|v_{2,n}|^q$ are all bounded. By $\mathcal{P}_\theta(v_{1,n},v_{2,n})\to 0$, we obtain $\{(v_{1,n},v_{2,n})\}$ is bounded in $H$. The remaining part is analogous to the step 1 in Lemma \ref{Lemma3.7紧性}.
			
			\textbf{Step 2:} We want to show $v_1 \neq 0$, $v_2 \neq 0$, then $v_1,v_2>0$ by the maximum principle.\\
			Analogous to the previous discussion, here we only need to exclude the case that $v_1=0$, $v_2=0$. Recall the proof process in Lemma \ref{Lemma3.9能量值}  $(\romannumeral1)$, $\inf\limits_{\mathcal{P}_\theta(\alpha_1,\alpha_2)}(|\nabla v_1|_2^2+|\nabla v_2|_2^2)>0$. Note that
			$\mathcal{P}_\theta(v_{1,n},v_{2,n})= 0$ and $\mathcal{P}_\theta(v_1,v_1)= 0$. Thus, it follows that
			\[\lim\limits_{n \to +\infty}(|\nabla v_{1,n}|_2^2+|\nabla v_{2,n}|_2^2)>\inf\limits_{\mathcal{P}_\theta(\alpha_1,\alpha_2)}(|\nabla v_1|_2^2+|\nabla v_2|_2^2)>0.\]
			Similar to step 2 in Lemma \ref{Lemma3.7紧性}, we obtain $\lim\limits_{n \to +\infty}(|\nabla v_{1,n}|_2^2+|\nabla v_{2,n}|_2^2) \geq S_H^{\frac{2_\omega^*}{2_\omega^*-1}}$. Then
			\[
				c=\lim\limits_{n \to +\infty}\mathcal{J}_\theta(v_{1,n},v_{2,n})\geq  \frac{2_\omega^*-1}{22_\omega^*}S_H^{\frac{2_\omega^*}{2_\omega^*-1}}-\varepsilon\alpha_1\alpha_2,
			\]
			which contradicts \eqref{3.19}.
			
			\textbf{Step 3:}  We prove that $(v_{1,n},v_{2,n})\to (v_1,v_2)$ in $H$. Since the proof method is the same as step 1 in Lemma \ref{Lemma3.7紧性}, we skip it.
	\end{proof}

	\begin{proof}[Proof of Theorem 1.2]\leavevmode\par
		We claim that for any $\eta >0$, one can find $\overline{\theta}=\overline{\theta}(\eta)>0$ satisfying
		\begin{equation}
			m_\theta(\alpha_1,\alpha_2) < \eta,  \qquad \text{for any} ~ \theta \geq \overline{\theta}.
			\label{3.20}
		\end{equation} 
		 For $v \in H^1(\mathbb{R}^N)\setminus \{0\}$, define 
		\[L(v)=\frac{1}{2}\int_{\mathbb{R}^N}|\nabla v|^2 - \frac{1}{2_\omega^*} \int_{\mathbb{R}^N}(I_\omega*|v|^{2_\omega^*})|v|^{2_\omega^*}.  \]
		By direct calculation, it follows that
		\[ \max_{t>0} L(t \diamond v)=\frac{2_\omega^*-1}{22_\omega^*}\frac{|\nabla v|_2^{\frac{22_\omega^*}{2_\omega^*-1}}}{(\int_{\mathbb{R}^N}(I_\omega*|v|^{2_\omega^*})|v|^{2_\omega^*})^{\frac{1}{2_\omega^*-1}}} .
		\]
	    Let $\varphi \in C_0^{\infty}(\mathbb{R}^N)$ with $|\varphi|_2 \leq \min \{ \alpha_1,\alpha_2\}$. By Lemma \ref{Lemma3.9能量值} $(\romannumeral2)$, we obtain
	    \begin{align*}
	    	m_\theta(\alpha_1,\alpha_2)&\leq m_\theta(|\varphi|_2 ,|\varphi|_2 ) \leq \max_{t>0} \mathcal{J}_\theta(t\diamond(\varphi,\varphi))\\
	    	&=\max_{t>0} \left( 2L(t\diamond \varphi)-\theta t^{\gamma_p+\gamma_q}\int_{\mathbb{R}^N}(I_\omega*|\varphi|^p)|\varphi|^q \right) - \varepsilon |\varphi|_2^2.
	    \end{align*}
	    
	    \par Note that $L(t\diamond \varphi) \to 0^+$ as $t \to 0^+$. Hence, one can find $\overline{t}>0$ such that for any $t<\overline{t}$, $L(t\diamond \varphi)<\eta$. Alternatively, we can find $\overline{\theta} >0$ satisfying
	    \[ \max\limits_{t \geq \overline{t}} \left( 2L(t\diamond \varphi)-\theta t^{\gamma_p+\gamma_q}\int_{\mathbb{R}^N}(I_\omega*|\varphi|^p)|\varphi|^q \right)<\eta,  \qquad \text{for all} ~ \theta \geq \overline{\theta}.
	    \]
	    Thus, \eqref{3.20} holds. Let $\overline{\varepsilon}:=\frac{1}{2\alpha_1\alpha_2}\frac{2_\omega^*-1}{22_\omega^*}S_H^{\frac{2_\omega^*}{2_\omega^*-1}}$. Now for $0<\varepsilon<\overline{\varepsilon}$ and $\eta_0=\frac{1}{2\alpha_1\alpha_2}\frac{2_\omega^*-1}{22_\omega^*}S_H^{\frac{2_\omega^*}{2_\omega^*-1}},$ by \eqref{3.20}, there exists $\theta_* >0$ such that
	    \[m_\theta(\alpha_1,\alpha_2) < \eta_0<\frac{2_\omega^*-1}{22_\omega^*}S_H^{\frac{2_\omega^*}{2_\omega^*-1}}-\varepsilon\alpha_1\alpha_2,  \qquad \text{for any} ~ \theta \geq \theta_*.
	    \]
	    We conclude from Lemma \ref{Lemma3.11 PS sequence} and Lemma \ref{Lemma3.12紧性} that system $\eqref{system1.1} $ has a normalized ground state $(v_1,v_2)$.
	\end{proof}

	\section*{Acknowledgements}
	The authors have no acknowledgements to declare.


\begin{thebibliography}{100}
\bibitem{Ambrosetti2007}A. Ambrosetti, E. Colorado, Standing waves of some coupled nonlinear Schrödinger equations. J. Lond. Math. Soc. (2) {\bf75} (2007), no. 1, 67–82.
\bibitem{Anderson1995}M. Anderson, J. Ensher, M. Matthews, C. Wieman, E. Cornell,  Observation of Bose-Einstein condensation in a dilute atomic vapor. science, 269.5221 (1995), 198-201.
\bibitem{Bartsch2010}T. Bartsch, N. Dancer, Z. Wang, A Liouville theorem, a-priori bounds, and bifurcating branches of positive solutions for a nonlinear elliptic system. Calc. Var. Partial Differential Equations {\bf37} (2010), no. 3-4, 345–361.
\bibitem{Bartsch2018}T. Bartsch and L. Jeanjean, Normalized solutions for nonlinear Schrödinger systems. Proc. Roy. Soc. Edinburgh Sect. A {\bf148} (2018), no. 2, 225–242.
\bibitem{Bartsch2016}T. Bartsch, L. Jeanjean, N. Soave, Normalized solutions for a system of coupled cubic Schrödinger equations on {$\Bbb{R}^3$}. J. Math. Pures Appl. (9) {\bf106} (2016), no. 4, 583–614.
\bibitem{bartsch2023existence}T. Bartsch, H. Li, W. Zou, Existence and asymptotic behavior of normalized ground states for Sobolev critical Schrödinger systems. Calc. Var. Partial Differential Equations {\bf62} (2023), no. 1, Paper No. 9, 34 pp.
\bibitem{Bartsch2019}T. Bartsch, N. Soave, Multiple normalized solutions for a competing system of Schrödinger equations. Calc. Var. Partial Differential Equations {\bf58} (2019), no. 1, Paper No. 22, 24 pp.
\bibitem{battaglia2018groundstates}L. Battaglia , J. Schaftingen, Groundstates of the Choquard equations with a sign-changing self-interaction potential. Z. Angew. Math. Phys. {\bf69} (2018), no. 3, Paper No. 86, 16 pp.
\bibitem{Chen2023}H. Chen, X. Yang, Prescribed mass solutions to Schrödinger systems with linear coupled terms. J. Geom. Anal. {\bf33} (2023), no. 11, Paper No. 347, 32 pp.
\bibitem{chen2021normalizedsolutionsnonlinearschrodinger}Z. Chen, X. Zhong, W. Zou, Normalized solutions for nonlinear Schrödinger systems with special mass-mixed terms: The linear couple case. arXiv 2107.12564.
\bibitem{chen2021normalized}Z. Chen, W. Zou, Normalized solutions for nonlinear Schrödinger systems with linear couples. J. Math. Anal. Appl. {\bf499} (2021), no. 1, Paper No. 125013, 22 pp.
\bibitem{Deconinck2004}B. Deconinck, P. Kevrekidis, H. Nistazakis, D. Frantzeskakis, Linearly coupled Bose-Einstein condensates: From Rabi oscillations and quasiperiodic solutions to oscillating domain walls and spiral waves. Physical Review A—Atomic, Molecular, and Optical Physics 70.6 (2004): 063605.
\bibitem{deng2025normalized}C. Deng, Z. Zhang, Normalized ground states to Sobolev critical Schrödinger systems with linear and nonlinear couplings. Discrete and Continuous Dynamical Systems, 2025: 0-0.
\bibitem{ghimenti2016nodal}M. Ghimenti, J. Schaftingen, Nodal solutions for the Choquard equation. J. Funct. Anal. {\bf271} (2016), no. 1, 107–135.
\bibitem{Han2011}Q. Han, F. Lin, Elliptic partial differential equations, volume 1. American Mathematical Soc., 2011.
\bibitem{ikoma2014compactness}N. Ikoma, Compactness of minimizing sequences in nonlinear Schrödinger systems under multiconstraint conditions. Adv. Nonlinear Stud. {\bf14} (2014), no. 1, 115–136.
\bibitem{Jeanjean1997}L. Jeanjean, Existence of solutions with prescribed norm for semilinear elliptic equations. Nonlinear Anal. {\bf28}(1997), no.10,1633-1659.
\bibitem{kato1972schrodinger}T. Kato, Schrödinger operators with singular potentials. Israel Journal of Mathematics {\bf13} (1972): 135-148.
\bibitem{Li2016}K. Li, Z. Zhang, Existence of solutions for a Schrödinger system with linear and nonlinear couplings. J. Math. Phys. {\bf57} (2016), no. 8, 081504, 17 pp.
\bibitem{Li2025}M. Li, J. He, H. Xu, M. Yang, Normalized solutions of linearly coupled Choquard system with potentials. Math. Methods Appl. Sci. {\bf48} (2025), no. 4, 4439–4459.
\bibitem{lieb2001analysis}E. Lieb, M. Loss, Analysis. Vol. 14. American Mathematical Soc., 2001.
\bibitem{lin2005ground}T. Lin, J. Wei, Ground state of {$N$} coupled nonlinear Schrödinger equations in {$\Bbb{R}^n$}, {$n\leq 3$}. Comm. Math. Phys. {\bf255} (2005), no. 3, 629–653.
\bibitem{Luo2019}H. Luo, Z. Zhang, Existence and nonexistence of bound state solutions for Schrödinger systems with linear and nonlinear couplings. J. Math. Anal. Appl. {\bf475} (2019), no. 1, 350–363.
\bibitem{Menyuk1987}C. Menyuk, Nonlinear pulse propagation in birefringent optical fibers. IEEE Journal of Quantum electronics 23.2 (1987): 174-176.
\bibitem{MOROZ2013153}V. Moroz, J. Schaftingen, Groundstates of nonlinear Choquard equations: existence, qualitative properties and decay asymptotics. J. Funct. Anal. {\bf265} (2013), no. 2, 153–184.
\bibitem{moroz2015existence}V. Moroz, J. Schaftingen, Existence of groundstates for a class of nonlinear Choquard equations. Trans. Amer. Math. Soc. {\bf367} (2015), no. 9, 6557–6579.
\bibitem{soave2020normalized}N. Soave, Normalized ground states for the NLS equation with combined nonlinearities. J. Differential Equations {\bf269} (2020), no. 9, 6941–6987.
\bibitem{Soave2020}N. Soave, Normalized ground states for the NLS equation with combined nonlinearities: the Sobolev critical case. J. Funct. Anal. {\bf279} (2020), no. 6, 108610, 43 pp.
\bibitem{Van2018}J. Schaftingen, J. Xia, Groundstates for a local nonlinear perturbation of the Choquard equations with lower critical exponent. J. Math. Anal. Appl. {\bf464} (2018), no. 2, 1184–1202.
\bibitem{Timmermans1998}E.Timmermans, Phase separation of Bose-Einstein condensates. Physical review letters 81.26 (1998): 5718.
\bibitem{Yang2013}M. Yang, Y. Wei, Existence and multiplicity of solutions for nonlinear Schrödinger equations with magnetic field and Hartree type nonlinearities. J. Math. Anal. Appl. {\bf403} (2013), no. 2, 680–694.
\bibitem{Yun2024}Z. Yun, Z. Zhang, Existence of normalized solutions for Schrödinger systems with linear and nonlinear couplings. Bound. Value Probl. 2024, Paper No. 25, 20 pp.
\bibitem{zhang2023normalizedsolutionscriticalchoquard}H. Zhang, J. Zhang, X. Zhong, Normalized solutions for critical Choquard systems. arXiv 2307.01483.
\end{thebibliography}
\end{document}